\documentclass[british,english]{amsart}
\usepackage{mathptmx}

\usepackage[T1]{fontenc}
\usepackage{textcomp}
\usepackage[latin9]{luainputenc}
\usepackage{babel}
\usepackage{amstext}
\usepackage{amsthm}
\usepackage{amssymb}
\usepackage[pdfusetitle,
 bookmarks=true,bookmarksnumbered=false,bookmarksopen=false,
 breaklinks=false,pdfborder={0 0 1},backref=false,colorlinks=false]
 {hyperref}

\makeatletter
\numberwithin{equation}{section}
\numberwithin{figure}{section}
\theoremstyle{plain}
\newtheorem{thm}{\protect\theoremname}[section]
\theoremstyle{definition}
\newtheorem{defn}[thm]{\protect\definitionname}
\theoremstyle{plain}
\newtheorem{lem}[thm]{\protect\lemmaname}
\theoremstyle{definition}
\newtheorem{example}[thm]{\protect\examplename}
\theoremstyle{plain}
\newtheorem{prop}[thm]{\protect\propositionname}

\@ifundefined{date}{}{\date{}}
\usepackage{tikz-cd}
\usepackage{tikz}

\makeatother

\addto\captionsbritish{\renewcommand{\definitionname}{Definition}}
\addto\captionsbritish{\renewcommand{\examplename}{Example}}
\addto\captionsbritish{\renewcommand{\lemmaname}{Lemma}}
\addto\captionsbritish{\renewcommand{\propositionname}{Proposition}}
\addto\captionsbritish{\renewcommand{\theoremname}{Theorem}}
\addto\captionsenglish{\renewcommand{\definitionname}{Definition}}
\addto\captionsenglish{\renewcommand{\examplename}{Example}}
\addto\captionsenglish{\renewcommand{\lemmaname}{Lemma}}
\addto\captionsenglish{\renewcommand{\propositionname}{Proposition}}
\addto\captionsenglish{\renewcommand{\theoremname}{Theorem}}
\providecommand{\definitionname}{Definition}
\providecommand{\examplename}{Example}
\providecommand{\lemmaname}{Lemma}
\providecommand{\propositionname}{Proposition}
\providecommand{\theoremname}{Theorem}

\begin{document}
\selectlanguage{british}%
\global\long\def\res{\!\restriction}%

\global\long\def\actson{\curvearrowright}%

\global\long\def\Ad{\text{Ad}}%

\global\long\def\Prob#1{\text{Prob}\left(#1\right)}%

\global\long\def\Ch#1{\text{Ch}\left(#1\right)}%

\global\long\def\Tr#1{\text{Tr}\left(#1\right)}%

\global\long\def\Trleq#1{\text{Tr}_{\leq1}\left(#1\right)}%

\global\long\def\normalizer#1#2{\text{N}_{#1}\left(#2\right)}%

\global\long\def\FC#1{\text{FC}\left(#1\right)}%

\global\long\def\Sub#1{\text{Sub}\left(#1\right)}%

\global\long\def\conn#1{#1^{0}}%

\global\long\def\explain#1#2{\underset{\underset{\mathclap{#2}}{\downarrow}}{#1}}%

\global\long\def\Ind#1#2{\text{Ind}_{#2}^{#1}}%

\global\long\def\Res#1#2{\text{Res}_{#2}^{#1}}%

\global\long\def\N{\mathbb{N}}%
\global\long\def\Z{\mathbb{Z}}%
\global\long\def\Q{\mathbb{Q}}%
\global\long\def\R{\mathbb{R}}%
\global\long\def\C{\mathbb{C}}%
\global\long\def\H{\mathbb{H}}%
\global\long\def\T{\mathbb{T}}%
\global\long\def\P{\mathbb{P}}%
\global\long\def\K{\mathbb{K}}%

\global\long\def\bG{\mathbf{G}}%
\global\long\def\bP{\mathbf{P}}%

\global\long\def\BH{\mathcal{B}(\mathcal{H})}%

\global\long\def\sl#1{\mathrm{SL}(#1)}%

\global\long\def\slz{SL_{n}\left(\mathbb{Z}\right)}%
\global\long\def\slr{SL_{n}\left(\mathbb{R}\right)}%

\global\long\def\conn#1{#1^{0}}%

\global\long\def\explain#1#2{\underset{\underset{\mathclap{#2}}{\downarrow}}{#1}}%

\title{Structural properties of reduced C{*}-algebras associated with higher-rank
lattices}
\author{itamar vigdorovich}
\thanks{This work was supported by NSF postdoctoral fellowship grant DMS-2402368.}
\begin{abstract}
We present the first examples of higher-rank lattices whose reduced
$C^{*}$-algebras satisfy strict comparison, stable rank one, selflessness,
uniqueness of embeddings of the Jiang--Su algebra, and allow explicit
computations of the Cuntz semigroup. This resolves a question raised
in recent groundbreaking work of Amrutam, Gao, Kunnawalkam Elayavalli,
and Patchell, in which they exhibited a large class of finitely generated
non-amenable groups satisfying these properties. Our proof relies
on quantitative estimates in projective dynamics, crucially using
the exponential mixing for diagonalizable flows $A\actson G/\Gamma$.
As a result, we obtain an effective mixed-identity-freeness property,
which, combined with V. Lafforgue's rapid decay theorem, yields the
desired conclusions.
\end{abstract}

\maketitle

\section{Introduction\protect\label{sec:Introduction}}

Let $\Gamma$ be a discrete group, and let $C_{r}^{*}(\Gamma)$ be
its reduced $C^{*}$-algebra, defined as the norm-closure of the algebra
generated by the regular representation $\lambda_{\Gamma}$. Understanding
the relationship between the group $\Gamma$ and the $C^{*}$-algebra
$C_{r}^{*}(\Gamma)$ is a fundamental problem that has been studied
extensively over the past several decades. The purpose of this paper
is to exhibit higher-rank lattices which satisfy the interesting group-theoretic
property introduced recently in \cite{2024strictcomparisonreducedgroup}
as \textit{selflessness}. Using this, we settle several open questions
regarding important structural properties of the associated reduced
$C^{*}$-algebras. For definitions and background, see the introduction
of \cite{2024strictcomparisonreducedgroup} and the references therein. 
\begin{thm}
\label{Thm:C*-algebras}Let $\Gamma$ be a cocompact lattice $\mathrm{PSL}_{3}(\K)$
where $\K$ is a local field of characteristic $0$ (thus $\K$ is
isomorphic to either $\R$, $\C$, or a finite field extension of
$\Q_{p}$). Then: 
\begin{enumerate}
\item The stable rank of $C_{r}^{*}(\Gamma)$ is $1$. 
\item $C_{r}^{*}(\Gamma)$ has strict comparison. 
\item $C_{r}^{*}(\Gamma)$ is selfless.
\item Any two embeddings of the Jiang-Su algebra $\mathcal{Z}$ into $C_{r}^{*}(\Gamma)$
are approximately unitarily conjugate. 
\item The Cuntz semigroup $\mathrm{Cu}(C_{r}^{*}(\Gamma))$ is isomorphic
to $\mathrm{V}(C_{r}^{*}(\Gamma))\sqcup[0,\infty]$ where $\mathrm{V}(C_{r}^{*}(\Gamma))$
is the Murray-von Neumann semigroup. 
\end{enumerate}
\end{thm}

Each of the statements in Theorem \ref{Thm:C*-algebras} is new. These
properties all play a major role in Elliott's classification program
for amenable (or nuclear) algebras, which has seen immense success
in recent years \cite{Rordam,MatuiSato,GonglinNiu1,GongLinNiu2,White2023}.

Beyond the amenable setting, however, much less is understood, and
only in the last few years has progress accelerated. Until recently,
the only non-amenable groups known to exhibit all of these powerful
properties were infinite free products \cite{dykema1998projections,`2023selfless}.
Various works established stable rank one for hyperbolic groups---culminating
in \cite{osinstablerank,raum2025twisted} covering all acylindrically
hyperbolic groups. Then, a breakthrough in \cite{2024strictcomparisonreducedgroup}
proved that non-elementary hyperbolic groups, graph products, mapping
class groups, and others, also satisfy each of the properties appearing
in \ref{Thm:C*-algebras}; remarkably, properties $(2),(3),(4),(5)$
were not known even in the case of the free group $F_{2}$!

Notably, the aforementioned examples arise from groups acting on hyperbolic
spaces, and indeed the methods crucially exploit these geometric assumptions.
By contrast, a prominent class of groups that do not fit into this
picture are lattices in simple Lie groups of real rank at least two.
In fact, such lattices do not act on hyperbolic spaces in any substantial
way (see, for instance, \cite{bader2022hyperbolic}). Consequently,
fundamentally new ideas are needed to move beyond this barrier. Motivated
by these challenges, the authors of \cite{2024strictcomparisonreducedgroup}
specifically asked for cocompact lattice in $\mathrm{SL}_{3}(\R)$
(see the discussion following Theorem B in \cite{2024strictcomparisonreducedgroup}).
Our Theorem \ref{Thm:C*-algebras} provides exactly such examples.\\

We now turn to discuss our group theoretic result that underlies Theorem
\ref{Thm:C*-algebras}. Let $F_{n}$ denote the free group on $n\in\N$
generators $x_{1},...,x_{n}$. An element of the free product $w\in\Gamma*F_{n}$
can be viewed as a formula in the \textit{variables} $x_{1},...,x_{n}$
and with \textit{coefficients} in $\Gamma$. Given $\gamma_{1},...,\gamma_{n}\in\Gamma$
we write $w(\gamma_{1},...,\gamma_{n})$ for the element in $\Gamma$
obtained by substituting $x_{i}$ by $\gamma_{i}$ and applying the
multiplication in $\Gamma$. If $w(\gamma_{1},...,\gamma_{n})=1$
for all $\gamma_{1},...,\gamma_{n}\in\Gamma$ then $w$ is said to
be a \textit{mixed identity }of $\Gamma$. For instance, the word
$x\gamma x^{-1}\gamma^{-1}\in\Gamma*\Z$ is a mixed identity if and
only if $\gamma$ is central. A group is called\textit{ mixed identity
free (MIF)} if it has no mixed identities except for the trivial one
$e\in\Gamma*F_{n}$. This property is significantly stronger than
having no \textit{identities} (in which case $w$ is required to belong
to $F_{n}$). It is not hard to see that the MIF property, as well
as its quantitative variant below, remains unchanged if one restricts
to one variable $n=1$. 

Given a MIF group $\Gamma$ and a non-trivial word $w\in\Gamma*F_{n}$,
one may ask ``how far'' in $\Gamma$ one must search for elements
that violate the equation $w(x_{1},...,x_{n})=1$. The following result
establishes that cocompact lattices in $\mathrm{PSL}_{d}$ are MIF
and, moreover, exhibit this property with a \textit{linear rate}.\footnote{The same result holds for $\mathrm{SL}_{d}$ as long as one requires
the coefficients to be non-central.} 
\begin{thm}
\label{Thm:effective=000020MIF}Fix $d\geq2$, and a local field\footnote{Any local field is isomorphic to either the reals $\R$, the $p$-adics
$\Q_{p}$, Laurent polnomials $\mathbb{F}_{p}((t))$, or to a finite
field extension of such.} $\K$. Given a cocompact lattice $\Gamma$ in $\mathrm{PSL}_{d}(\K)$,
and a fixed finite generating set $S$ on it, there is a constant
$C=C(\Gamma,S)>1$ such that the following holds. For any $r>0$ and
$n\in\N$ there exists elements $\gamma_{1},...,\gamma_{n}\in\Gamma$
such that
\begin{enumerate}
\item $w(\gamma_{1},...,\gamma_{n})\ne e$ for any $e\ne w\in\Gamma*F_{n}$
whose coefficients are of word-length at most $r$, and,
\item The word-length of each $\gamma_{i}$ is at most $C\cdot r$. 
\end{enumerate}
\end{thm}

To illustrate the significance of Theorem \ref{Thm:effective=000020MIF},
note the following most basic consequence: For any $r>0$, there exists
an element $\gamma$ of word-length at most $C\cdot r$ which commutes
with no element of $\Gamma$ of length $\leq r$ other than the trivial
element. Different choices of $w$ lead to different consequences
of this sort. Indeed, see \cite{hull2016transitivity} for other implications
of MIF.

It is impossible to obtain a sublinear word-length bound, or even
a linear bound with $C=1$. This is because, e.g., the word $\gamma x^{-1}\in\Gamma*\Z$
is satisfied by $\gamma$. In this sense, this result is optimal.
We do not attempt to estimate the constant $C$, but we suspect that
in the case $d\geq3$ it depends only on $G$, not on $\Gamma$. If
$d=2$, the lattice $\Gamma$ is hyperbolic, in which case similar
statements are shown in \cite{2024strictcomparisonreducedgroup} with
polynomial rather than linear bounds. This was improved to a linear
bound in the new result \cite{bradford2025non}. \\

After posting the first version of this work, we learned of several
recent and ongoing relevant results. In \cite{bradford2024length},
arithmetic methods are used to obtain linear MIF bounds for Zariski
dense subgroups that are defined over their trace fields. In an ongoing
work, Avni and Gelander are developing a strengthened super-approximation
theorem which will show that once a linear group is MIF, linear bounds
follow. Finally, Becker and Breuillard have announced their comprehensive
concentration-of-measure framework for algebraic subvarieties of semisimple
groups which, in particular, implies linear MIF bounds. 

We note that the results cited above produce, for each $r$, an element
of length $\leq Cr$ that violates only \textit{finitely} many words,
namely those $w\in\Gamma*\Z$ whose word length with respect to $S\cup\{x\}$
is $\leq r$. By contrast, Theorem \ref{Thm:effective=000020MIF}
yields an element that violates all words $w\in\Gamma*\Z$ whose \textit{coefficients}
each have word length $\leq r$ (with respect to $S$), with no bound
on the length of $w$. In particular, \textit{infinitely} many constrains
are violated simultaneously. This difference reflects our dynamical
approach:
\begin{enumerate}
\item We develop effective versions of key ingredients in the Tits Alternative.
This is useful for other ``quantitative freeness'' results in linear
groups, including work in progress.
\item We exploit the exponential mixing of the geodesic flow (or, in higher-rank,
the torus flow), established by Kleinbock and Margulis.
\end{enumerate}
A detailed outline of the argument appears at the end of this section.
\\

Theorem \ref{Thm:C*-algebras} is deduced from Theorem \ref{Thm:effective=000020MIF}
as follows. The quantitative version of the MIF property established
in Theorem \ref{Thm:effective=000020MIF} is a strengthening of the
notion of a \textit{selfless group} introduced in \cite{2024strictcomparisonreducedgroup}
(see Definition \ref{def:selfless} and Lemma \ref{lem:selfless}
below). The key idea in \cite{2024strictcomparisonreducedgroup} (inspired
by \cite{louder2025strongly}) is that if a group is both selfless,
and satisfies the \textit{rapid decay} property, then $C_{r}^{*}(\Gamma)$
is a\textit{ selfless $C^{*}$-algebra} as introduced by Robert \cite{`2023selfless}.
The insight in \cite{`2023selfless} is that selfless $C^{*}$-algebras
mimic infinite free products from a model-theoretic point of view;
as a result, methods from free probability can be employed to obtain
strict comparison along with the other structural properties appearing
in Theorem \ref{Thm:C*-algebras}.

Thus, the rapid decay property is a crucial ingredient for our proof
of Theorem \ref{Thm:C*-algebras}. This property was initially studied
by Haagerup \cite{HaagerupGod} and later by Jolissaint \cite{Jolissaint},
and is tightly connected to the Baum-Connes conjecture, see \cite{valette2002introduction,sapir2015rapid,chatterji2017introduction}.
While known to hold for hyperbolic groups and certain generalizations,
the rapid decay property remains obscure in the realm of lattices
in semisimple Lie groups. It was shown in \cite{lafforgue} that cocompact
lattices in $\mathrm{SL}_{3}(\R)$ and $\mathrm{SL}_{3}(\C)$ satisfy
the rapid decay property, following the results of \cite{ramagge1998haagerup}
covering the non-Archimedean case. This is precisely the reason why
in Theorem \ref{Thm:C*-algebras} we restrict to the group $\mathrm{SL}_{3}$.
Yet, Valette conjectured that all cocompact lattices in semisimple
Lie groups have the rapid decay property \cite[Conjecture 7]{valette2002introduction}.
Despite much effort, this conjecture remains open except in some special
cases \cite{ramagge1998haagerup,lafforgue,chatterji2003property},
and perhaps, the present available $C^{*}$-algebraic implications
may serve as an additional encouragement towards a resolution. 

In contrast, non-cocompact irreducible lattices in higher-rank semisimple
Lie groups never satisfy the rapid decay property. A natural question
arising from our work is whether non-cocompact higher-rank lattices,
e.g. $\mathrm{SL}_{3}(\Z)$, satisfy the properties given in Theorem
\ref{Thm:C*-algebras} (or even just one them). In the absence of
the rapid decay property, a new idea is required. 

\subsection*{Outline}

The proof of Theorem \ref{Thm:effective=000020MIF} follows the foundational
framework laid by Tits, where it is shown that every finitely generated
linear group is either virtually solvable or contains a free subgroup
\cite{tits1972free} (see \cite{cohen2024invitation} for a complete
and elegant proof). The method for identifying a free subgroup is
to consider a suitable linear representation, and seek for a pair
of elements which ``play ping-pong'' in the projective action. This
influential result has been extended in many directions, including
\cite{bestvina2000tits,sageev2005tits,breuillard2007topological,bars2024tits},
to name just a few. See also \cite{eskin2001uniform,breuillard2008uniform}
for effective results of a different flavor than those in this paper.

In our setting, however, we are not seeking free subgroups per se.
Instead, we want a single element that is \textquotedblleft free\textquotedblright{}
from all elements in the ball $B_{\Gamma}(r)$ of radius $r$ in $\Gamma$
(though not necessarily from products of those elements). This is
reminiscent of the ``simultaneous ping-pong pairs'' problem \cite{bekka1994some},
but we are unaware of any implication between the two. Section \ref{sec:Free-independences}
formalizes the relevant notion of freeness along with suitable ping-pong
lemma.

In Section \ref{sec:Non-effective-freeness} we apply \textquotedblleft soft\textquotedblright{}
arguments to obtain (non-quantitative) the MIF property for centre-free
Zariski dense subgroups of $\mathrm{SL}_{d}$. This result is well
known, but we present this short proof as an exposition for the rest
of the article. 

To achieve a quantitative result, we must carefully track dynamical
and geometric properties with explicit bounds. Thus, in Section \ref{sec:Freeness-criteria-for},
we formulate a criterion (Lemma \ref{lem:dynamical-criterion}) for
an element $g\in G$ to be free from $B_{\Gamma}(r)$. 

Even finding an element $g\in G$--let alone an element $\gamma\in\Gamma$--that
is of bounded size and is free from $B_{\Gamma}(r)$, poses a challenge.
The reason is that balls in $\Gamma$ are of exponential cardinality,
leading to exponentially many geometric constraints to satisfy effectively.
The approach we take is geometric and probabilistic: we consider a
single element $h\in\Gamma$ and show that, with exponentially high
probability, a ``generic'' configuration of points and hyperplanes
(later to become attracting and repelling loci) is well positioned
relative to $h$. A union bound then guarantees the existence of a
single configuration that works for all $h\in B_{\Gamma}(r)$, from
which we deduce the existence of the desired element $g\in G$. This
is covered in Section \ref{sec:finding=000020r=000020free=000020element=000020in=000020G},
and is where most of the technical difficulty of this paper lies. 

The next challenge is to replace $g$ with a lattice element $\gamma\in\Gamma$.
The Zariski density of $\Gamma$ in $G$ used in the non-quantitative
analysis can be made quantitative (e.g via \cite[Proposition 3.2]{eskin2001uniform}),
but in a way which results in an exponential blow-up (again, due to
the exponential number of constraints). Another standard approach
is to use Poincaré recurrence to find a power $g^{t}$ near a lattice
point, but maintaining full quantitative control, requires $g^{t}$
to be exponentially close to a lattice point, which typically takes
an exponential amount of time $t$. 

We overcome this by defining a function $\psi:G\to\R_{\geq0}$ that
measures how effectively an element $g\in G$ ``plays ping-pong''
with elements of the ball $B_{\Gamma}(r)$ (Definition \ref{def:geometric=000020function}).
Crucially, $\psi$ is Lipschitz and invariant under the subgroup of
diagonal matrices $A\leq G$. In Section \ref{sec:find=000020r-free=000020element=000020in=000020Gamma},
we consider the homogeneous dynamical system $A\actson G/\Gamma$
and apply Kleinbock-Margulis's exponential mixing theorem (and its
extensions) to the function $\psi$. It follows that, after a short
time, the $A$-flow starting near $g\Gamma\in G/\Gamma$ will pass
exponentially close to $e\Gamma\in G/\Gamma$, the identity coset.
Exploiting both the Lipschitz property and the $A$-invariance of
$\psi$, we then extract a lattice element $\gamma\in\Gamma$ of linear
size whose $\psi$-value is approximately that of $g$. This final
step completes the proof of Theorem \ref{Thm:effective=000020MIF}. 

Section \ref{sec:selfless-reduced--algebras} is devoted to spelling
out the details needed to deduce Theorem \ref{Thm:C*-algebras} from
Theorem \ref{Thm:effective=000020MIF}.

\subsection*{Acknowledgments }

This work would not have been possible without the support of many
individuals. I would like to thank Tsachik Gelander for his guidance,
Sri Kunnawalkam Elayavalli for his continuous encouragement and advice,
and Gregory Patchell for insightful discussions. I extend my gratitude
to Amir Mohammadi for his help regarding effective results in homogeneous
dynamics, as well as to Andreas Wieser for valuable correspondence
on the subject. I am very grateful to Tianyi Zheng and Ido Grayevski
for their thought-provoking conversations, and for their affirmations.
A special thanks is extended to Uri Bader and to Adrian Ioana for
past and present mentorship.

\section{Free independences\protect\label{sec:Free-independences}}

Let $\Gamma$ be a discrete group with a fixed generating set $S\subset\Gamma$.
Denote by \textbf{$B_{\Gamma}(r)$} the ball of radius $r$ with respect
to the word-length associated with $S$. Fix a generator $x$ for
the infinite cyclic group $\Z$, and consider the free product $\Gamma*\Z$.
Then $S\cup\{x\}$ is a generating set for $\Gamma*\Z$ and we denote
by $B_{\Gamma*\Z}(r)$ the corresponding $r$-ball. The following
notion was introduced in \cite[Definition 3.1]{2024strictcomparisonreducedgroup}
\begin{defn}
\label{def:selfless}A group $\Gamma$ endowed with a finite generating
set $S\subset\Gamma$ is \textit{selfless} if there is a function
$f:\N\to\R$ with $\liminf_{r}f(r)^{1/r}=1$ such that the following
holds: for any $r\in\N$, there is an epimorphism $\phi_{r}:\Gamma*\Z\to\Gamma$
with $\phi_{r}\mid_{\Gamma}=\mathrm{Id}_{\Gamma}$, $\phi_{r}$ is
injective on $B_{\Gamma*\Z}(r)$, and $\phi_{r}\left(B_{\Gamma*\Z}(r)\right)\subset B_{\Gamma}(f(r))$. 
\end{defn}

There is an equivalent way to understand this definition. An element
$w\in\Gamma*\Z$ can be written uniquely as
\begin{equation}
w=g_{1}\cdot...\cdot g_{l}\label{eq:word}
\end{equation}
for $l\in\N\cup\{0\}$, and where each $g_{i}$ is a non-trivial element
of either $\Gamma$ or $\Z$, with $g_{i}\in\Gamma\iff g_{i+1}\in\Z$.
Each $g_{i}$ that belongs to $\Gamma$ is referred to as a \textit{coefficient}
of $w$. An element $g_{i}$ which belongs to $\Z$ may be written
as $g_{i}=x^{k_{i}}$ for some $k_{i}\in\Z\backslash\{0\}$. We refer
to $x$ as the \textit{variable} of $w$. We refer to $w$ as a word;
it is said to be trivial if it is the identity element of $\Gamma*\Z$,
or equivalently if $l=0$ in (\ref{eq:word}). For $\gamma\in\Gamma$
we denote $w(\gamma)$ the element of $\Gamma$ obtained by replacing
each appearance of $x$ with $\gamma$, and applying the multiplication
law in $\Gamma$. We say that $\gamma$ \textit{satisfies} $w$ if
$w(\gamma)=e$, and \textit{violates} otherwise. 
\begin{defn}
\label{def:r-free}Let $\Gamma$ be a group.
\begin{enumerate}
\item An element $\gamma\in\Gamma$ is said to be \textit{freely independent}
from a set $F\subset\Gamma$ if it violates any non-trivial word $w\in\Gamma*\Z$
all of whose coefficients lie in $F\backslash\{e\}$. 
\item Suppose $\Gamma$ is endowed with a length function $\Gamma\to\N\cup\{0\}$.
Then $\gamma\in\Gamma$ is said to be $r$-free, for some $r>0$,
if $\gamma$ is freely independent from the set of element in $\Gamma$
of length at most $r$.
\end{enumerate}
\end{defn}

Note that any $\gamma\in\Gamma$ gives rise to the homomorphism 
\[
\phi:\Gamma*\Z\to\Gamma,\qquad w\mapsto w(\gamma),
\]
and conversely, from any homomorphism $\phi:\Gamma*\Z\to\Gamma$ we
get an element $\gamma=\phi(x)$. This sets a bijection between $\Gamma$
and the set of epimorphisms $\phi:\Gamma*\Z\to\Gamma$ satisfying
$\phi\mid_{\Gamma}=\mathrm{Id}_{\Gamma}$. Moreover, if an element
$\gamma\in\Gamma$ is $r$-free, then the corresponding map $\phi$
is injective on the ball of radius $r$. Indeed, if $w,w'\in B_{\Gamma*\Z}(r)$
satisfy $\phi(w)=\phi(w')$ then $\left(w^{-1}w'\right)(\gamma)=\phi(w^{-1}w')=e$
which implies that $w^{-1}w'$ is trivial because $\gamma$ is $r$-free.
Furthermore, it is not hard to see that $\phi(B_{\Gamma*\Z}(r))\subset B_{\Gamma}(r|\gamma|)$.
The following lemma therefore follows. 
\begin{lem}
\label{lem:selfless}Let $\Gamma$ be a group endowed with some finite
generating set $S$, and let $f:\N\to\R$ be a function satisfying
$\liminf_{r}f(r)^{1/r}$. Suppose that for any $r\geq0$ there is
an $r$-free element $\gamma_{r}\in\Gamma$ with $|\gamma_{r}|\leq f(r)$.
Then $\Gamma$ is selfless.
\end{lem}

We proceed with the following variant of the ping-pong lemma where
one ``player'' is an arbitrary set of elements, and the other ``player''
is a cyclic subgroup. 
\begin{lem}[Ping-Pong Lemma]
\label{prop:pingpong} Let $F$ be a subset of a group $\Gamma$,
and let $\gamma\in\Gamma$ of infinite order. Suppose that $\Gamma$
acts on some set $P$, and that there exists two non-empty disjoint
subsets $A$ and $B$ of $P$, such that, $\gamma^{k}.B\subset A$
for all $k\in\Z\backslash\{0\}$, and, $h.A\subset B$ for all $h\in F\backslash\{e\}$.
Then $\gamma$ is freely independent from $F$. 
\end{lem}

\begin{proof}
The proof is very standard, but due to the new terminology we give
the argument in full detail. 

Consider a non-trivial word $w\in\Gamma*\Z$ whose coefficients all
belong to $F\backslash\{e\}$. We must show that the corresponding
element of $w(\gamma)\in\Gamma$ is non-trivial. Write $w=g_{1}\cdot...\cdot g_{l}$,
as in (\ref{eq:word}). We can reduce to the case where $w$ starts
and ends with $x$, i.e when $g_{1},g_{l}\in\Z\backslash\{0\}$. Indeed,
for $g\in\Z\backslash\{0\}$ denote its sign by $\mathrm{sgn}(g)\in\{\pm1\}$,
and replace $w$ by some conjugate as follows:
\begin{itemize}
\item $g_{1}\in\Z\quad\&\quad g_{l}\in\Z:\qquad w\mapsto w$
\item $g_{1}\in\Z\quad\&\quad g_{l}\in\Gamma:\qquad w\mapsto x^{\mathrm{sgn}(g)}wx^{-\mathrm{sgn}(g)}$
\item $g_{1}\in\Gamma\quad\&\quad g_{l}\in\Z:\qquad w\mapsto x^{-\mathrm{sgn}(g)}wx^{\mathrm{sgn}(g)}$ 
\item $g_{1}\in\Gamma\quad\&\quad g_{l}\in\Gamma:\qquad w\mapsto x^{-1}wx$
\end{itemize}
In any case, this conjugation of $w$ does not affect whether $w(\gamma)$
is the identity or not. Thus, having replaced $w$ by this conjugation
we see $w(\gamma)$ starts and ends with a non-zero power of $\gamma$.

Now, we claim that $w(\gamma)\ne e$ by showing that $w(\gamma).B\subset A$.
This follows by induction: each application of $g_{i}$ alternates
between $B$ and $A$, by assumption. Since $g_{1}\in\Z\backslash\{0\}$,
we will indeed end up in $A$. 
\end{proof}

\section{Zariski dense subgroups are MIF\protect\label{sec:Non-effective-freeness}}

This section is devoted to proving Theorem \ref{Thm:effective=000020MIF}.
Some of the statements are subsumed by analogous quantitative statements
appearing later on. Since the repetition is mild, we do include the
details for the sake of completeness and as an exposition for the
following sections. 

Let $\mathbb{K}$ be a local field with absolute value $|\cdot|$.
Let $V$ be a $\mathbb{K}$-vector space of finite dimension $d\geq2$,
and let $P=P(V)$ denote the projective space. The topology on $\mathbb{K}$
induces a topology on the $\mathbb{K}$-algebraic variety $P$. This
is the only topology on $P$ that we consider. In the following sections
we will consider a natural metric on $P$, but for now, any metric
$d_{P}$ on $P$ which induces the topology on $P$ is appropriate.
Given $v\in V\backslash\{0\}$ we denote by $[v]=\mathbb{K}v$ the
corresponding point of $P$. 

For a general metric space $(X,d_{X})$, a set $X_{0}\subset X$,
and some $\epsilon>0$, we denote the $\epsilon$-tubular neighborhood
by 
\[
\left\{ X_{0}\right\} _{\epsilon}=\left\{ x\in X:\exists x_{0}\in X_{0}\text{ such that }d_{X}(x,x_{0})<\epsilon\right\} 
\]
Given two sets $X_{1},X_{2}$ in $X$, the notation $d_{X}(X_{1},X_{2})$
will always refer to the \textit{infimum} of the distance $d_{X}(x_{1},x_{2})$
among all points $x_{1}\in X_{1}$ and $x_{2}\in X_{2}$.

Let $\mathrm{SL}(V)$ denote the group of linear operators on $V$
with determinant $1$. If $V=\K^{d}$ we simply write $\mathrm{SL}_{d}(\K)=\mathrm{SL}(V)$.
This group acts on $V$, and thus on $P$, continuously. The kernel
of this action consists precisely of scalar operators in $\mathrm{SL}(V)$.
In the dynamics of $\mathrm{SL}(V)$ on $P$, there is essentially
one example (for our concerns) that requires good understanding. 
\begin{example}
\label{exa:exa-1}Let $g\in\mathrm{SL}_{d}(\K)$ be a diagonal matrix
of the form $g=\mathrm{diag}(\lambda_{1},...,\lambda_{d})$ with $|\lambda_{1}|\gneq|\lambda_{2}|\geq...\geq|\lambda_{d}|>0$.
Let $V_{\lambda}\leq\K^{d}$ denote the eigenspace associated to the
eigenvalue $\lambda$. Set
\[
p=V_{\lambda_{1}},\qquad W=\sum_{i=2}^{d}V_{\lambda_{i}},
\]
so that $p$ is a point in $P$, and $W$ is a hyperplane in $P$
not containing $p$. It is not hard to verify that $\lim_{k\to\infty}g^{k}.x=p$
for all $x\in P\backslash W$. Moreover, this convergence is uniform
on compact subsets \cite[Lemma 3.8]{tits1972free}, namely,
\begin{equation}
\max_{x\in C}d_{P}(g^{k}.x,p)\to0\qquad\text{for all}\quad C\subset P\backslash W\text{ compact}.\label{eq:proximality}
\end{equation}
\end{example}

\begin{defn}
An element $g\in\mathrm{SL}(V)$ is said to be \textit{proximal} if
there exists a point $p\in P$ and a hyperplane $W\subset P$ for
which (\ref{eq:proximality}) holds. $p$ is called the\textit{ attracting
point} of $g$, and $W$ is called the \textit{repelling hyperplane}
of $g$. 
\end{defn}

\begin{lem}
\label{lem:non-effective=000020existence=000020of=000020configuration}For
any countable set $F\subset\mathrm{SL}(V)$ consisting of non-scalar
elements, there exists an element $g\in\mathrm{SL}(V)$ such that
both $g$ and $g^{-1}$ are proximal, and moreover, 
\begin{equation}
\left[\bigcup_{h\in F}\left\{ h.p_{+},h.p_{-}\right\} \right]\cap\left(W_{+}\cup W_{-}\right)=\emptyset,\label{eq:non-effective-MIF=000020good=000020configuration}
\end{equation}
 where $p_{+}\in P(V)$ (and $p_{-}\in P(V)$) is the attracting point
of $g$ (resp. $g^{-1}$), and, $W_{+}\subset P(V)$ (and $W_{-}\subset P(V)$)
is the repelling hyperplane of $g$ (resp. $g^{-1}$).
\end{lem}

\begin{proof}
We will first choose the geometric configuration $W_{+},W_{-}$ and
$p_{+},p_{-}$, and then choose $g$ accordingly. The existence of
a configuration which satisfies (\ref{eq:non-effective-MIF=000020good=000020configuration})
is proven in Lemma \ref{lem:existence=000020of=000020points=000020and=000020subspace}
in a quantitive manner. Note that the assumptions there are stricter,
but they are used only to get quantitive estimates. The non-quantitive
version is significantly simpler, so we briefly spell out the details. 

Let $\left(p_{-},W_{+}\right)$ and $\left(p_{+},W_{-}\right)$ be
two flags of type $\left(1,d-1\right)$ in $V$ chosen randomly and
independently-- we refer to (\ref{eq:flag}) for the precise meaning
of this statement. Fix $h\in F$. Since $p_{+}$ and $W_{-}$ are
independent, $h.p_{+}\notin W_{+}$ almost surely. Similarly, $h.p_{-}\notin W_{-}$
almost surely. Moreover, since $h$ is non-scalar, its fixed point
set in $P$ is of measure zero. Therefore, even though $p_{-}\in W_{+}$,
almost surely we have that $h.p_{-}\notin W_{+}$. Similarly, $h.p_{+}\notin W_{-}$,
almost surely. 

We have thus considered countably many desired events, each of which
occurs almost surely. It follows that Equation (\ref{eq:non-effective-MIF=000020good=000020configuration})
occurs almost surely. Furthermore, the intersection of $W_{+}$ and
$W_{-}$ is almost surely a co-dimension $2$ subspace of $V$ which
we denote by $W_{0}$. Moreover, the points $p_{+}$ and $p_{-}$
do not belong to this subspace $W_{0}$, almost surely. 

Fix such $\left(p_{-},W_{+}\right)$ and $\left(p_{+},W_{-}\right)$
with the aforementioned generic properties. Let $g\in\mathrm{End}(V)$
be the unique diagonalizable linear operator with eigenvalues $2$,
$1$, and $\frac{1}{2}$, with corresponding eigenspaces $p_{+}$,
$W_{0}$ and $p_{-}$, respectively. Then clearly, $g\in\mathrm{SL}(V)$.
Moreover, both $g$ and $g^{-1}$ are proximal with corresponding
attracting points $p_{+}$ and $p_{-}$ and repelling hyperplanes
$W_{+}$ and $W_{-}$, respectively. The statement thus follows. 
\end{proof}
\begin{lem}
\label{lem:G=000020is=000020MIF}Let $G=\mathrm{SL}(V)$ . For any
finite set $F$ of non-scalar elements of $G$, there exists an element
$g\in G$ which is free from $F$. In particular, if $G$ is center-free
then it is MIF (considered as an abstract group). 
\end{lem}

\begin{proof}
Let $F\subset G$ be a finite set of non-scalar elements. Fix an element
$g\in G$ as provided in Lemma \ref{lem:non-effective=000020existence=000020of=000020configuration}.
Since the sets $W_{+}\cup W_{-}$ and $\bigcup_{h\in F}\left\{ h.p_{+},h.p_{-}\right\} $
are closed and disjoint, there exist $\epsilon_{1},\epsilon_{2}>0$
sufficiently small so that the tubular neighborhoods $\left\{ W_{+}\cup W_{-}\right\} _{\epsilon_{1}}$
and $\bigcup_{h\in F}\left\{ h.p_{+},h.p_{-}\right\} _{\epsilon_{2}}$
remain disjoint. Even more so, by making $\epsilon_{1}$ smaller if
necessary, we may assume that the following condition holds for all
points $x\in P$ and for all $h\in F$:
\begin{equation}
d_{P}(x,p_{\bullet})<\epsilon_{1}\Longrightarrow d_{P}(h.x,h.p_{\bullet})<\epsilon_{2},\qquad\left(\bullet\in\{+,-\}\right)\label{eq:implication}
\end{equation}
Indeed, this follows from $F$ being finite along with continuity
of the action of $G$ on $P$. 

Set
\[
A=\left\{ p_{+},p_{-}\right\} _{\epsilon_{1}}\qquad\text{and}\qquad B=\bigcup_{h\in F}\left\{ h.p_{+},h.p_{-}\right\} _{\epsilon_{2}}.
\]
Note that $A$ and $B$ are disjoint because $\{p_{+},p_{-}\}\subset W_{+}\cup W_{-}$. 

Now, $g$ and $g^{-1}$ are proximal, and thus, we there exists $k_{0}\in\N$
sufficiently large so that, for all $k$ with $|k|\geq k_{0}$,
\[
g^{k}.\left(P\backslash\left\{ W_{+}\cup W_{-}\right\} _{\epsilon_{1}}\right)\subset\left\{ p_{+},p_{-}\right\} _{\epsilon_{2}}=A.
\]
Fix such $k_{0}$, and set $\tilde{g}=g^{k_{0}}$. Since $B$ and
$\left\{ W_{+}\cup W_{-}\right\} _{\epsilon_{1}}$ are disjoint we
have on the one hand
\[
\tilde{g}^{k}.B\subset A,\qquad\forall k\in\Z\backslash\{0\}.
\]
On the other hand, given $h\in F$ and $x\in A$ we have that $d_{P}(x,p_{+})<\epsilon_{1}$
(or $d_{P}(x,p_{-})<\epsilon_{1}$ in which case the argument is the
same). Therefore $d_{P}(h.x,h.p_{+})<\epsilon_{2}$ by (\ref{eq:implication}),
which means that $h.x\in B$. This shows that 
\[
h.A\subset B,\qquad\forall h\in F.
\]
The ping-pong lemma (Lemma \ref{prop:pingpong}) thus applies, and
we get that $\tilde{g}$ is freely independent from $F$.
\end{proof}
\begin{thm}
\label{thm:non-effective=000020MIF=000020non-intro=000020version=000020}Let
$G=\mathrm{SL}(V)$, and let $\Gamma\leq G$ be a Zariski dense subgroup.
Then for any finite collection of words $\Omega\in G*\Z\backslash\{e\}$
with non-central coefficients, there exists $\gamma\in\Gamma$ such
that $w(\gamma)\ne0$ for all $w\in\Omega$. In particular, if $\Gamma$
is center-free, then it is MIF.
\end{thm}

\begin{proof}
For $w\in G*\Z$ let 
\[
\mathrm{Null}(w)=\left\{ g\in G:w(g)=1\right\} 
\]
Then $\mathrm{Null}(w)$ is a Zariski closed subset of $G$. By Lemma
\ref{lem:G=000020is=000020MIF}, it is a proper subset assuming $w\ne e$,
of strictly lower dimension because $G$ is Zariski connected. Thus,
given a finite subset of $\Omega\subset G*\Z\backslash\{e\}$, the
set $\bigcup_{w\in\Omega}\mathrm{Null}(w)$ is a proper Zariski closed
subset of $G$. In particular, $\Gamma$ is not contained in it. This
means that there exists $\gamma\in\Gamma$ such that $w(\gamma)\ne e$
for all $w\in\Omega$, as required. 
\end{proof}

\section{Freeness criteria for linear groups\protect\label{sec:Freeness-criteria-for}}

Let $\mathbb{K}$ be a local field with absolute value $|\cdot|$.
Let $V$ be $\mathbb{K}$-vector space of finite dimension $d\geq2$,
and fix an arbitrary identification $V=\K^{d}$. Endow $V$ with a
norm, as follows:
\begin{enumerate}
\item If $\K$ is Archimedean, i.e., if $\K$ is isomorphic to $\R$ or
to $\C$, then endow $V$ with the standard inner product and the
corresponding norm. 
\item If $\K$ is non-Archimedean, then endow $V$ with the $\infty$-norm,
$\|v\|=\max_{i=1}^{d}\left|v_{i}\right|$ for all $v\in V$. 
\end{enumerate}
Let $P=P(V)$ be the projective space. The norm on $V$ induces a
norm on the exterior product $V\wedge V$, and as a result, a metric
on $P$: 
\[
d_{P}([v_{1}],[v_{2}])=\frac{\|v_{1}\wedge v_{2}\|}{\|v_{1}\|\cdot\|v_{2}\|},\qquad v_{1},v_{2}\in V\backslash\{0\}.
\]
For example, in the Archimedean case, $d_{P}([v_{1}],[v_{2}])=\left|\sin\angle(v_{1},v_{2})\right|$
where $\angle(v_{1},v_{2})$ is the angle between the lines $[v_{1}]$
and $[v_{2}]$. 

The group $G=\mathrm{SL}(V)$ naturally acts on $P$, and more generally,
on the Grassmannian $\mathrm{Gr}_{k}(V)$ consisting of all $k$-dimensional
subspaces of $V$. The subgroup of $G$ which preserves the norm on
$V$ is a maximal compact subgroup $K\leq G$, and it acts transitively
on $P$. For example, in the case $\K=\R$ then $K=\mathrm{SO}(V)$,
the group of rotations. Note that $K$ preserves the metric on $P$. 

Let $\Gamma\leq G$ be a lattice. The goal of this Section \ref{sec:Freeness-criteria-for}
is to construct a function $\psi_{r}:G\to\R_{\geq0}$ which measures
how effectively does an element $g\in G$ ``play ping-pong'' with
the ball of radius $r$ in $\Gamma$. 

\subsection{Length function on $G$ }

The group $G$ acts on the homogeneous space $G/K$ by left translations.
In the case $\K$ is Archimedean, $X$ admit a $G$-invariant Riemannian
metric which makes it a symmetric space. For example, if $G=\mathrm{SL}_{2}(\R)$
then $G/K\cong\mathbb{H}^{2}$, the hyperbolic plane. In the non-Archimedean
Case, $X$ admits a combinatorial structure called the Bruhat-Tits
building. In any case, $G/K$ is admits a $G$-invariant metric, and
for $g\in G$ we define
\[
|g|=d_{G/K}(gK,eK)
\]
This is a pseudo-length function in the sense that $|e|=0$, $|g^{-1}|=|g|$
and $|gh|\leq|g|+|h|$ for all $g,h\in G$. Moreover, $|k_{1}gk_{2}|=|g|$
for all $k_{1},k_{2}\in K$. This in turn defines a $G$-invariant
pseudo-metric by $d(g,h)=|h^{-1}g|$. When restricted to a cocompact
lattice $\Gamma$ in $G$, this metric is quasi-isometric to the word
metric arising from any choice of a finite generating set for $\Gamma$.
However, the metric that comes from $G/K$ is preferable because it
is defined on all of $G$. For $r>0$, we denote $B_{G}(r)=\left\{ g\in G:|g|\leq r\right\} $,
the ball of radius $r$. We denote $B_{\Gamma}(r)=B_{G}(r)\cap\Gamma$,
and $B_{\Gamma}^{\circ}(r)=B_{\Gamma}(r)\backslash\mathrm{Z}(\Gamma)$
(where $\mathrm{Z}(\Gamma)$ denotes the center of $\Gamma$). 

The length function $|\cdot|$ should not be confused with the operator
norm on $G\subset\mathrm{End}(V)$ which is induced from the norm
on $V$:
\[
\|g\|=\max_{v\in V\backslash\{0\}}\frac{\|gv\|}{\|v\|}.
\]
The two are connected by\label{eq:equivalent-lengths}
\begin{equation}
\log\|g\|\leq|g|\leq\sqrt{d}\log\|g\|,
\end{equation}
see \cite[Lemma 4.5]{breuillard2008uniform}. 

Importantly, $G$ does not preserve the metric on $P$-- see example
\ref{exa:exa-1}. Nevertheless, each $g\in G$ is Lipschitz, and we
denote by $\mathrm{Lip}(g)$ its Lipschitz constant. In fact, the
following bound holds.
\begin{lem}
\label{lem:Lip}For any $g\in G$: 
\[
\mathrm{Lip}(g)\leq e^{4|g|}.
\]
 
\end{lem}

\begin{proof}
It is not hard to see that $\mathrm{Lip}(g)\leq\|g\|^{2}\|g^{-1}\|^{2}$
(see \cite[§2.1]{breuillard2008uniform}) and so by (\ref{eq:equivalent-lengths})
we get that $\mathrm{Lip}(g)\leq e^{4|g|}$. 
\end{proof}

\subsection{Dynamical criterion for $r$-freeness}

The following is a more restricted version of the notion of proximality
given in Section \ref{sec:Non-effective-freeness}.
\begin{defn}
\label{def:contractive}An element $g\in G$ is said to be \textit{very
proximal} if it is conjugate over the algebraic closure $\overline{\K}$
to a diagonal matrix of the form $\mathrm{diag}(\lambda_{1},...,\lambda_{d})$
with 
\[
|\lambda_{1}|\gneq|\lambda_{2}|\geq...\geq\left|\lambda_{d-1}\right|\gneq\left|\lambda_{d}\right|>0.
\]
 We denote by $G_{\mathrm{vp}}$ the set of all very proximal elements
of $G$. 

Clearly, if $g\in G_{\mathrm{vp}}$ then $g^{k}\in G_{\mathrm{vp}}$
for all $k\in\Z\backslash\{0\}$. Using the notation in Definition
\ref{def:contractive}, we define 
\[
\mathrm{Att:\quad}G_{\mathrm{vp}}\to P,\qquad\mathrm{Att}(g)=V_{\lambda_{1}}
\]
\[
\mathrm{Rep}\quad:G_{\mathrm{vp}}\to\mathrm{Gr}_{d-1}(V),\qquad\mathrm{Rep}(g)=\sum_{i=2}^{d}V_{\lambda_{i}}
\]
where $V_{\lambda_{i}}$ is the eigenspace corresponding to the eigenvalue
$\lambda_{i}$. Recall that $\mathrm{Gr}_{d-1}(V)$ denotes the space
of all co-dimension $1$ subspaces of $V$. We define the sets 
\[
\mathrm{Att}^{\pm}(g):=\left\{ \mathrm{Att}(g),\mathrm{\mathrm{Att}}(g^{-1})\right\} ,\quad\text{and}\quad\mathrm{Rep}^{\pm}(g):=\mathrm{Rep}(g)\cup\mathrm{Rep}(g^{-1})
\]
called the \textit{attracting locus} and \textit{repelling locus}
of $g$, respectively. Observe that
\end{defn}

\begin{enumerate}
\item $\mathrm{Att}(g^{k})=\mathrm{Att}(g)$ and $\mathrm{Rep}(g^{k})=\mathrm{Rep}(g)$
for all $k\in\N$. 
\item $\mathrm{Att}(hgh^{-1})=h.\mathrm{Att}(g)$ and $\mathrm{Rep}(hgh^{-1})=h.\mathrm{Rep}(g)$
for all $h\in G$.
\item $\mathrm{Att}(g)\in\mathrm{Rep}(g^{-1})$. 
\end{enumerate}
Given $g\in G_{\mathrm{vp}}$ and $r>0$, the following parameters
will play a central role in determining whether $g$ is $r$-free. 
\begin{itemize}
\item The \textit{contraction} parameter
\[
C_{g}=\max\left\{ \left|\lambda_{1}\right|/\left|\lambda_{2}\right|,\left|\lambda_{d-1}\right|/\left|\lambda_{d}\right|\right\} >1.
\]
Note that $C_{g^{k}}=C_{g}^{k}$ for any $k\in\N$, and that $C_{hgh^{-1}}=C_{g}$
for all $h\in G$.
\item The \textit{Lipschitz} parameter, which depends only $r$, is 
\[
L_{r}=\max_{h\in B_{G}(r)}\mathrm{Lip}(h).
\]
Note that $L_{r}\leq e^{4r}$ by Lemma \ref{lem:Lip}. 
\item The \textit{geometric} parameter
\[
D_{g,r}=\min_{h\in B_{\Gamma}^{\circ}(r)}d_{P}\left(h.\mathrm{Att^{\pm}}(g),\mathrm{Rep}^{\pm}(g)\right).
\]
 Note that $D_{g^{k},r}=D_{g,r}$ for all $k\in\Z\backslash\{0\}$.
This parameter depends on the lattice $\Gamma\leq G$.
\end{itemize}
The three parameters are related to $r$-freeness, as follows.
\begin{lem}
\label{lem:dynamical-criterion}If $g\in G_{\mathrm{vp}}$ satisfies
$D_{g,r}\geq(1+L_{r})C_{g}^{-1/2}$ then $g$ is $r$-free. 
\end{lem}

\begin{proof}
By \cite[Proposition 3.3]{breuillard2003dense} 
\[
g^{k}.\left\{ P\backslash\mathrm{Rep}^{\pm}(g)\right\} _{\epsilon_{g}}\subset\left\{ \mathrm{Att}^{\pm}(g)\right\} _{\epsilon_{g}}
\]
where $\epsilon_{g}=C_{g}^{-1/2}$, and for all $k\in\Z\backslash\{0\}$. 

Set $A=\left\{ \mathrm{Att}^{\pm}(g)\right\} _{\epsilon_{g}}$ and
$B=\left\{ \bigcup_{h\in B_{\Gamma}^{\circ}(r)}h.\mathrm{Att}^{\pm}(g)\right\} _{L_{r}\epsilon_{g}}$.
By assumption, $\mathrm{Rep}^{\pm}(g)$ and $\bigcup_{h\in B_{\Gamma}^{\circ}(r)}h.\mathrm{Att}^{\pm}(g)$
are $(1+L_{r})\epsilon_{g}$-apart. Since $\mathrm{Att}^{\pm}(g)\subset\mathrm{Rep}^{\pm}(g)$
we see that $A$ and $B$ are disjoint. 

On the one hand, given $x\in B$, we have by the triangle inequality
that
\begin{align*}
d_{P}(x,\mathrm{Rep}^{\pm}(g) & )\geq d_{P}\left(\bigcup_{h\in B_{\Gamma}^{\circ}(r)}h.\mathrm{Att}^{\pm}(g),\mathrm{Rep}^{\pm}(g)\right)-d_{P}\left(x,\bigcup_{h\in B_{\Gamma}^{\circ}(r)}h.\mathrm{Att}^{\pm}(g)\right)\\
 & \geq D_{g,r}-L_{r}\epsilon_{g}\geq(1+L_{r})\epsilon_{g}-L_{r}\epsilon_{g}=\epsilon_{g}
\end{align*}
Since $g$ is $\epsilon_{g}$-contracting we get that $g^{k}x\in A$
for all $k\ne0$. 

On the other hand, given $x\in A$ we have that $d_{P}(x,\mathrm{Att}(g))<\epsilon_{g}$
(or $d_{P}(x,\mathrm{Att}(g^{-1}))<\epsilon_{g}$ in which case the
argument is the same). Therefore for any $h\in B_{\Gamma}^{\circ}(r)$
\[
d_{P}(hx,h\mathrm{Att}(g))<L_{r}\cdot\epsilon_{g},
\]
which means that $hx\in B$. The ping-pong lemma (Lemma \ref{prop:pingpong})
thus applies.
\end{proof}

\subsection{Geometric criterion}

Among the three parameters appearing in Lemma \ref{lem:dynamical-criterion},
it is the geometric parameter $D_{g,r}$ that is the most difficult
to control. We shall now formulate a criterion expressed only in terms
of this parameter. Certainly, $G$ admits very proximal and regular
(i.e., diagonalizable with distinct eigenvalues) elements. Using Poincare
recurrence, it is not hard to see that $\Gamma$ also admits such
elements. 
\begin{defn}
\label{def:geometric=000020function}Fix a very proximal element $a_{0}\in\Gamma$.
For any $r>0$, define $\psi_{r}^{a_{0}}=\psi_{r}:G\to\R$ by 
\[
\psi_{r}(g)=D_{ga_{0}g^{-1},r}=\min_{h\in B_{\Gamma}^{\circ}(r)}d_{P}(hg\mathrm{Att}^{\pm}(a_{0}),g\mathrm{Rep}^{\pm}(a_{0})).
\]
We call $\psi_{r}$ the \textit{geometric function}.
\end{defn}

Our goal can be expressed purely in terms of the function $\psi_{r}$. 
\begin{lem}
\label{lem:restate=000020goal}Assume that there exists $c,\kappa>0$,
such that for any $r\in\N$ there exists $\gamma\in\Gamma$ satisfying
\begin{equation}
|\gamma|<cr,\qquad\text{and}\qquad\psi_{r}(\gamma)>e^{-\kappa r}.\label{eq:gamma=000020requirement}
\end{equation}
Then, for any $r\in\N$ there exists an $r$-free element $\tilde{\gamma}\in\Gamma$
with $|\tilde{\gamma}|\leq c'r$, where $c'$ does not depend on $r$. 
\end{lem}

\begin{proof}
Fix $r\in\N$, and let $\gamma\in\Gamma$ such that \ref{eq:gamma=000020requirement}
holds. Denote $l=\log C_{a_{0}}$ and set $q=\left\lceil \frac{5+\kappa}{l}\right\rceil \in\N$.
Consider the element 
\[
\tilde{\gamma}=\gamma a_{0}^{2qr}\gamma^{-1}\in\Gamma.
\]
Then we have 
\[
C_{\tilde{\gamma}}^{-1/2}=\left(C_{a_{0}}^{2qr}\right)^{-1/2}=e^{-lqr},\qquad D_{\tilde{\gamma},r}=D_{\gamma a_{0}\gamma^{-1},r}>e^{-\kappa r},\qquad L_{r}\leq e^{4r}.
\]
Therefore
\[
C_{\tilde{\gamma}}^{-1/2}\left(1+L_{r}\right)\leq e^{5r-lqr}\leq e^{-\kappa r}<D_{\tilde{\gamma},r}.
\]
It follows from Lemma \ref{lem:dynamical-criterion} that $\tilde{\gamma}$
is $r$-free. Additionally
\[
|\tilde{\gamma}|\leq|\gamma|+|\gamma^{-1}|+2qr|a_{0}|\leq(2c+2q|a_{0}|)r.
\]
This demonstrates for any $r>0$ an $r$-free element $\tilde{\gamma}$
with $|\tilde{\gamma}|\leq c'r$ where $c':=2c+2q|a_{0}|$ does not
depend on $r$. 
\end{proof}
We will need two more properties of the function $\psi_{r}$, both
are crucial for our analysis.  The first property follows immediately
by definition. 
\begin{lem}
\label{lem:A-invariance}For any $r$, the function $\psi_{r}$ is
right-invariant under the centralizer $A=C_{G}(a_{0})\leq G$, namely
\[
\psi_{r}(ga)=\psi_{r}(g),\qquad\text{for all }g\in G,\quad a\in A.
\]
\end{lem}

The second property says that $\psi_{r}$ is Lipschitz continuous
in the following sense:
\begin{lem}
\label{lem:holder=000020continuitiy}For any $g,s\in G$ with $\|s-1\|<\frac{1}{2}$
we have
\[
|\psi_{r}(sg)-\psi_{r}(g)|\leq4e^{4r}\|s-1\|.
\]
\end{lem}

\begin{proof}
Let $v\in V$ with $\|v\|=1$. Then
\[
\|sv\wedge v\|=\|(s-1)v\wedge v\|\leq\|s-1\|.
\]
In addition, 
\[
\left|\|sv\|-1\right|=\left|\|sv\|-\|v\|\right|\leq\|sv-v\|\leq\|s-1\|
\]
which implies that 
\[
\|sv\|\geq1-\|s-1\|.
\]
We get that 
\[
d_{P}(sv,v)=\frac{\|sv\wedge v\|}{\|sv\|}\leq\frac{\|s-1\|}{1-\|s-1\|}\leq2\|s-1\|,
\]
where the last inequality uses that $\|s-1\|<\frac{1}{2}$. Hence,
given any $h\in B_{\Gamma}^{\circ}(r)$ and any $[v_{1}],[v_{2}]\in P$,
we get by the triangle inequality that
\begin{align*}
\left|d_{P}\left(hsg[v_{1}],sg[v_{2}]\right)-d_{P}(hg[v_{1}],g[v_{2}])\right| & \leq d_{P}(hsg[v_{1}],hg[v_{1}]+d_{P}(sg[v_{2}],g[v_{2}])\\
 & \leq\mathrm{Lip}(h)\cdot d_{P}(sg[v_{1}],g[v_{1}])+d_{P}(sg[v_{2}],g[v_{2}])\\
 & \leq2\left(1+e^{4r}\right)\|s-1\|\\
 & \leq4e^{4r}\|s-1\|.
\end{align*}
This in particular applies to any $[v_{1}]\in\mathrm{Att}^{\pm}(a_{0})$
and any $[v_{2}]\in\mathrm{Rep}^{\pm}(a_{0})$. Taking the infimum
over all such combinations, we get that 
\[
\left|\psi_{r}(sg)-\psi_{r}(g)\right|\leq4e^{4r}\|s-1\|.
\]
\end{proof}

\section{finding $r$-free elements in $G$\protect\label{sec:finding=000020r=000020free=000020element=000020in=000020G}}

In the previous section we have defined the geometric function $\psi_{r}:G\to\R_{\geq0}$
(see Definition \ref{def:geometric=000020function}). This current
section is devoted to finding $g\in G$, near the identity, such that
$\psi_{r}(g)\geq ce^{-\kappa r}$, for some constants $c,\kappa>0$. 

In what follows, $\K$ is a local field, and $V(\K)$ is a $\K$-vector
space of dimension $d\geq2$. If $\K'$ is a finite field extension
of $\K$, then we denote $V(\K')=V(\K)\otimes_{\K}\K'$ the corresponding
$\K'$-vector space. The absolute value on $\K$ extends uniquely
to an absolute value on $\K'$, and so, any norm on $V(\K)$ extends
to a norm on $V(\K')$. 

We will often consider an operator $h\in\mathrm{End}(V(\K))$ which
is diagonalizable, but only over the algebraic closure $\overline{\K}$.
We denote by $\K_{h}$ the splitting field of the characteristic polynomial
of $h$. 

We fix an identification $V(\K)=\K^{d}$ with the norm described in
Section \ref{sec:Freeness-criteria-for}. Let $P(\mathbb{K})=P(V(\mathbb{K}))$
the projective space, endowed with the corresponding metric. Note
that $P(\mathbb{K})\subset P(\K')$ is an isometric embedding, and
so, in such situation, by $d_{P}$ we will always mean the metric
defined over $P(\K')$. If, however, the context is such that there
is only one field under consideration, then we shall often write $V=V(\K)$
and $P=P(\K)$. 

The maximal compact subgroup $K\leq G=\mathrm{SL}(V)$ acts transitively
on the unit sphere of $V$ and on $P$. Thus, there exists a unique
$K$-invariant Borel probability measure on the unit sphere of $V$
as well as on $P$, and more generally on the Grassmannian $\mathrm{Gr}_{k}(V)$.
We refer to such a probability measure as the \textit{uniform} measure.
In what follows, we will considered random unit vectors $v\in V$,
as well as the corresponding random points $[v]\in P$, all with respect
to this uniform measure. 

\subsection{Basic inequalities on projective space}

In the following few lemmas, we will use the notation $C_{d}$ whenever
we are referring to a constant which depends only on the dimension
$d$. This $C_{d}$ is not necessarily the exact same constant in
the various lemmas.

We start with the following basic geometric fact.
\begin{lem}
\label{lem:tubular-neihbothood} Let $\K'$ be a finite field extension
of $\K$, and denote $V=V(\K')$. For any $\epsilon>0$, and for a
random unit vector $v\in V(\mathbb{K})$, the following holds:
\begin{enumerate}
\item For any non-trivial $\K'$-linear subspace $W\leq V$,
\[
\bP\left[d_{P}([v],[W])<\epsilon\right]<C_{d}\epsilon
\]
\item Let $V=W_{1}\oplus W_{2}$ be a non-trivial direct sum decomposition.
Write $v=a_{1}w_{1}+a_{2}w_{2}$ for unit vectors $w_{1}\in W_{1}$,
$w_{2}\in W_{2}$, and $a_{1},a_{2}\in\K'$. Then 
\[
\bP\left[|a_{1}|<\epsilon\right]<C_{d}\epsilon
\]
\end{enumerate}
\end{lem}

\begin{proof}
If $W=V$ then the statement is clear. Assume that $W$ is proper.
Treating $V$ as a $\K$-vector space, we may fix some $\K$-subspace
$\tilde{W}\leq V$ with $\dim_{\K}\tilde{W}=\dim_{\K}V-1$ and which
contains $W$. It suffices to prove the statement for $\tilde{W}$
rather then for $W$. Now this boils down to computing the volume
of the tubular neighborhood $\left\{ [\tilde{W}]\right\} _{\epsilon}$
inside $P$, and it is not hard to see that this volume is at most
$C_{d}\epsilon$ for some dimensional constant $C_{d}$. This shows
the first statement. 

Note that for any unit vector $w_{2}'\in W_{2}$,
\[
d_{P}([v],[w_{2}'])=\|v\wedge w_{2}'\|=|a_{1}|\cdot\|w_{1}\wedge w_{2}\|\leq|a_{1}|.
\]
Therefore $d_{P}\left([v],[W_{2}]\right)\leq|a_{1}|$. The second
statement thus follows from the first statement. 
\end{proof}
In the non-effective proof of the MIF property given in Section \ref{sec:Non-effective-freeness},
we implicitly used the fact that the set of fixed points in $P$ under
a non-scalar matrix $h\in\mathrm{SL}(V)$ is a proper subspace and
in particular a measure $0$ set. Lemma \ref{lem:almost-fixed-points}
below is a quantitative version of this statement. For that, we will
first need several linear algebraic inequalities. 
\begin{lem}
\label{lem:linear=000020algebra}Let $h\in\mathrm{End}(V)$. Let $\lambda_{1},\lambda_{2}\in\K$
be two distinct eigenvalues of $h$ and let $v_{1}$ and $v_{2}$
be two unit eigenvectors corresponding to $\lambda_{1}$ and $\lambda_{2}$.
Then
\begin{equation}
\|v_{1}\wedge v_{2}\|\geq\frac{|\lambda_{1}-\lambda_{2}|}{\|h\|+|\lambda_{1}-\lambda_{2}|}.\label{eq:v1=000020wedge=000020v2}
\end{equation}
In addition, if $v\in V$ is of the form $v=a_{1}v_{1}+a_{2}v_{2}$
for some $a_{1},a_{2}\in\K$, then 
\begin{equation}
\|hv\wedge v\|\geq\min\{|a_{1}|,|a_{2}|\}^{2}\cdot\frac{|\lambda_{1}-\lambda_{2}|^{2}}{\|h\|+|\lambda_{1}-\lambda_{2}|}.\label{eq:norm=000020of=000020hv=000020wedge=000020v}
\end{equation}
\end{lem}

\begin{proof}
Let $V_{\lambda_{1}},V_{\lambda_{2}}$ denote the eigenspaces corresponding
to $\lambda_{1}$ and $\lambda_{2}$. Let $v_{1}\in V_{\lambda_{1}}$
and $v_{2}\in V_{\lambda_{2}}$ both of norm $1$. Up to conjugating
$h$ be an element $k\in K$, we may assume that $V_{\lambda_{1}}=\mathrm{Sp}\{e_{1},...,e_{k}\}$
for some $1\leq k<d$. We may write $v_{2}=\alpha^{-1}\left(tw_{1}+w_{2}\right)$
for some unit vector $w_{1}\in V_{\lambda_{1}}$, some unit vector
in $\mathrm{Sp}\{e_{k+1},...,e_{d}\}$, some $t\in\K$, and with $\alpha=\|tw_{1}+w_{2}\|\leq|t|+1$.
Then
\[
hw_{2}=\alpha hv_{2}-thw_{1}=\alpha\lambda_{2}v_{2}-t\lambda_{1}w_{1}=t\lambda_{2}w_{1}+\lambda_{2}w_{2}-t\lambda_{1}w_{1}.
\]
 It follows that 
\[
\|h\|\geq\|hw_{2}\|\geq\|t\lambda_{2}w_{1}-t\lambda_{1}w_{1}\|=|t|\cdot|\lambda_{2}-\lambda_{1}|.
\]
Therefore
\begin{equation}
\alpha^{-1}\geq\frac{1}{|t|+1}\geq\frac{|\lambda_{2}-\lambda_{1}|}{\|h\|+|\lambda_{2}-\lambda_{1}|}.\label{eq:lamba1-lambda2}
\end{equation}
 In addition, 
\begin{align}
\|v_{1}\wedge v_{2}\| & =\alpha^{-1}\|v_{1}\wedge tw_{1}+v_{1}\wedge w_{2}\|\geq\alpha^{-1}\|v_{1}\wedge w_{2}\|=\alpha^{-1}.\label{eq:||v1=000020x=000020v2||}
\end{align}
The first statement follows from Equations (\ref{eq:lamba1-lambda2},\ref{eq:||v1=000020x=000020v2||}).
The second statement follows from the first statement as follows:
\begin{align*}
\|hv\wedge v\| & =\left|a_{1}\right|\left|a_{2}\right|\left|\lambda_{1}-\lambda_{2}\right|\cdot\|v_{1}\wedge v_{2}\|\geq\min\{|a_{1}|,|a_{2}|\}^{2}\frac{|\lambda_{1}-\lambda_{2}|^{2}}{\|h\|+|\lambda_{1}-\lambda_{2}|}.
\end{align*}
\end{proof}
For $h\in\mathrm{End}(V)$ diagonalizable, and let $\sigma(h)\subset\K_{h}$
be the set of eigenvalues of $h$ over its splitting field. Assume
that $h$ is non-scalar so that it acts non-trivially on $P$. The
degree to which $h$ deviates from being a scalar operator is measured
by
\[
\theta(h)=\max_{\lambda_{1},\lambda_{2}\in\sigma(h)}\left|\lambda_{1}-\lambda_{2}\right|.
\]
This, however, does not take into account any geometric properties
between eigenvectors of $h$. In light of Lemma \ref{lem:linear=000020algebra}
we define
\[
\omega(h)=\sqrt{\frac{\|h\|+\theta(h)}{\theta(h)^{2}}}
\]

\begin{lem}
\label{lem:almost-fixed-points} Let $h\in\mathrm{End}(V(\K))$ diagonalizable
over $\K_{h}$ and non-scalar. Then for a random unit vector $v$
in $V(\mathbb{K})$ we have
\[
{\bf P}\left[\|hv\wedge v\|<\epsilon\right]<C_{d}\omega(h)\sqrt{\epsilon}.
\]
\end{lem}

\begin{proof}
Recall that over the splitting field, any operator is triangulizable
by an element of $K$. By this we mean that there exists $k\in K(\K_{h})$
such that the matrix representing $h$ over the basis $ke_{1},...,ke_{d}\in V(\K_{h})$
is upper triangular. Here $e_{1},...,e_{d}$ denote the standard basis
of $V=\K^{d}$. In the Archimedean case, this is just to say that
any operator in unitarily triangulizable over $\C$, and in the non-Archimedean
case, there is an analogous statement, see \cite[Proposition 4.5.2]{bump1998automorphic}.
We may thus assume that $h$ is a triangular matrix with entries in
$\K_{h}$. Moreover, it is possible to choose the basis so that $\left|h_{11}\right|=\max_{\lambda\in\sigma(h)}|\lambda|$
and $\left|h_{22}\right|=\min_{\lambda\in\sigma(h)}|\lambda|$. Let
$\tilde{h}$ be the upper left $2$-by-$2$ corner of $h$ and note
that $\|\tilde{h}\|\leq\|h\|$. Moreover, the eigenvalues of $\tilde{h}$
are exactly $h_{11}$ and $h_{22}$ which are distinct (because $h$
is non-scalar). In particular, $\tilde{h}$ is diagonalizable and
$\theta(\tilde{h})=\theta(h)$. 

Fix unit eigenvectors $v_{1}$ and $v_{2}$ corresponding to the eigenvalues
$h_{11}$ an $h_{22}$ of $\tilde{h}$. Let $v\in V(\mathbb{K})$
be a random unit vector, and write $v=a_{1}v_{1}+a_{2}v_{2}+a_{3}e_{3}...+a_{d}e_{d}$
with $a_{i}\in\K_{h}$. Let $\tilde{v}=a_{1}v_{1}+a_{2}v_{2}$, the
projection of $v$ onto $\mathrm{Sp}_{\K_{h}}\{e_{1},e_{2}\}$. Assume
that $|a_{1}|$ and $|a_{2}|$ are at least $\omega(h)\sqrt{\epsilon}$;
this occurs with probability $\geq1-2C_{d}\omega(h)\sqrt{\epsilon}$
because of Lemma \ref{lem:tubular-neihbothood}. Then, by Lemma \ref{lem:linear=000020algebra},
\[
\|hv\wedge v\|\geq\|\tilde{h}\tilde{v}\wedge\tilde{v}\|\geq\min\left\{ |a_{1}|,|a_{2}|\right\} ^{2}\cdot\frac{\theta(h)^{2}}{\|h\|+\theta(h)}=\min\left\{ |a_{1}|,|a_{2}|\right\} ^{2}\cdot\omega(h)^{-2}=\epsilon.
\]
We now consider the following flag variety 
\begin{equation}
\mathcal{F}(V)=\mathcal{F}_{1,d-1}(V)=\left\{ \left([v],[W]\right):[v]\in[W]\right\} \subset P(V)\times\mathrm{Gr}_{d-1}(V)\label{eq:flag}
\end{equation}
The diagonal action of $K$ on $P(V)\times\mathrm{Gr}_{d-1}(V)$ preserves
the compact subset $\mathcal{F}$, and acts on it transitively. As
a result, we get a unique $K$-invariant probability measure on $\mathcal{F}$.
It can be understood as follows: we first choose $W\in\mathrm{Gr}_{d-1}(V)$
uniformly at random, and then choose $v$ to be a random unit vector
in $W$. 
\end{proof}
\begin{lem}
\label{lem:random=000020orthogonal=000020points-preliminary=000020lemma}Denote
$V=V(\K)$. Let $h\in\mathrm{End}(V)$ diagonalizable over $\K_{h}$
and non-scalar. Let $([v],[W])\in\mathcal{F}(V)$ random. Given $\epsilon>0$,
we have
\[
\bP\left[d_{P}([hv],[W])<\epsilon\right]<C_{d}(1+\omega(h))\|h\|^{1/4}\epsilon^{1/4}
\]
\end{lem}

\begin{proof}
We start by giving a geometric qualitative proof that is valid in
the Archimedean case. This proof can be made quantitive, but we will
omit the details because, in a moment, we will give a formal proof
that is valid in the general case. 

Let $([v],[W])\in\mathcal{F}(V)$ random, with $\|v\|=1$. By Lemma
\ref{lem:almost-fixed-points}, with high probability, $[hv]$ does
not remain close to $[v]$. Let $\widetilde{hv}$ be the projection
of $hv$ onto the orthogonal complement of $v$ in $V$. Let $W_{0}$
be the orthogonal complement of $v$ inside $W$. It follows from
Lemma \ref{lem:tubular-neihbothood} that, with high probability,
$\left[\widetilde{hv}\right]$ does not lie in $\left[W\right]$,
nor is it very close to $\left[W\right]$. In particular, $[hv]$
is not very close to $[W]$. 

We now proceed with the general case. Fix $\delta=\sqrt{\|h\|\epsilon}$.
Let $k$ be a random element of $K$. Denote $v=ke_{1}$, $W_{0}=\mathrm{Sp}\{ke_{2},...,ke_{d-1}\}$,
$W=\mathrm{Sp}\{v,W_{0}\}$ and $u=ke_{d}$. Write 
\[
hv=a_{1}v+a_{2}w_{0}+a_{3}u,\qquad a_{1},a_{2},a_{3}\in\K
\]
where $w_{0}\in W_{0}$ unit vector. Denote $\widetilde{hv}=a_{2}w_{0}+a_{3}u$
so that $\|\widetilde{hv}\|=\|hv\wedge v\|$. Note that $v$ is a
random unit vector in $V$ and therefore, by Lemma \ref{lem:almost-fixed-points},
\begin{equation}
\bP\left[\|\widetilde{hv}\|\geq\delta\right]\geq1-C_{d}\omega(h)\sqrt{\delta}.\label{eq:||hv||=00005Cgeqdelta}
\end{equation}
Moreover, for any fixed $v$ that satisfies $\|\widetilde{hv}\|\geq\delta$,
we have by Lemma \ref{lem:tubular-neihbothood}
\begin{equation}
\bP\left[d_{P}\left(\left[\widetilde{hv}\right],\left[W_{0}\right]\right)\geq\delta\right]\geq1-C_{d}'\delta.\label{eq:d=00005CgeqCddelta}
\end{equation}
Note that in the lemma, the vector is random while the subspace if
fixed, but since the metric is $K$-invariant this makes no difference.
This shows the following bound on the conditional probability:
\begin{equation}
\bP\left[d_{P}\left(\left[\widetilde{hv}\right],\left[W_{0}\right]\right)\geq\delta\mid\|\widetilde{hv}\|\geq\delta\right]\geq1-C_{d}'\delta\label{eq:conditional}
\end{equation}
It follows from (\ref{eq:||hv||=00005Cgeqdelta}), (\ref{eq:conditional})
that

\begin{align}
\bP\left[\|\widetilde{hv}\|\geq\delta\text{ and }d_{P}\left(\left[\widetilde{hv}\right],\left[W_{0}\right]\right)\geq\delta\right] & \geq\left(1-C_{d}\omega(h)\sqrt{\delta}\right)\left(1-C_{d}'\delta\right)\label{eq:bias}\\
 & \geq1-C_{d}\omega(h)\sqrt{\delta}-C_{d}'\delta\nonumber \\
 & \geq1-C_{d}''(1+\omega(h))\sqrt{\delta}\nonumber \\
 & =1-C_{d}''(1+\omega(h))\|h\|^{1/4}\epsilon^{1/4}\nonumber 
\end{align}

It is left to show that once the event (\ref{eq:bias}) occurs, one
has $d_{P}\left([hv],[W]\right)\geq\epsilon$. From $\|\widetilde{hv}\|\geq\delta$
and $d_{P}\left(\left[\widetilde{hv}\right],\left[W_{0}\right]\right)\geq\delta$
we have that
\[
|a_{3}|=\|\widetilde{hv}\wedge w_{0}\|\geq\delta\|\widetilde{hv}\|\geq\delta^{2}
\]
 It follows that 
\[
d_{P}([hv],[w])=\frac{\|hv\wedge w\|}{\|hv\|\cdot\|w\|}\geq\frac{\|a_{3}u\wedge w\|}{\|h\|\cdot\|w\|}\geq\frac{\|w\|}{\|h\|\cdot\|w\|}\delta^{2}=\epsilon
\]
 and therefore $d_{P}([hv],[W])\geq\epsilon$. 
\end{proof}

\subsection{Uniform displacement for lattices}

Recall that if a topological group $H$ acts continuously on a metric
space $Y$ by isometries, then the corresponding displacement function
is defined by
\[
\mathrm{disp}_{Y}:H\to\R_{\geq0},\qquad h\mapsto\inf_{y\in Y}d_{Y}(h.y,y)
\]
The map $\mathrm{disp}_{Y}$ is clearly conjugation invariant. Moreover,
if $h_{n}\to e$, then $\mathrm{disp}_{Y}(h_{n})\to0$. 

An element $h$ is called \textit{semisimple} if the infimum is realized,
i.e. if $\mathrm{disp}_{Y}(h)=d_{Y}(h.y,y)$ for some $y\in Y$; it
is moreover called \textit{elliptic} if $d_{Y}(h.y,y)=0$ and \textit{hyperbolic}
if $d_{Y}(h.y,y)>0$. If it is not semisimple, it is called \textit{parabolic}. 
\begin{lem}
\label{lem:spectrum=000020of=000020elements=000020in=000020cocompact=000020lattice}Let
$\Gamma$ be a cocompact lattice in $G=\mathrm{SL}(V)$ , and suppose
that a sequence $h_{n}\in\Gamma$ satisfies $\theta(h_{n})\to0$.
Then $h_{n}\in\mathrm{Z}(\Gamma)$ for almost all $n$. 
\end{lem}

\begin{proof}
Note that $\Gamma\cdot\mathrm{Z}(G)$ is again a cocompact lattice
of $G$, so we may assume that $\mathrm{Z}(G)\subset\Gamma$. Recall
that we have fixed an identification $V(\K)=\K^{d}$. Any element
in a cocompact lattice is semisimple \cite[Corollary 4.31]{morris2015introduction}.
Let $h_{n}\in\Gamma$ satisfying $\theta(h_{n})\to0$, and let $\K_{n}=\K_{h_{n}}$
denote the splitting field. Then there exists a diagonal matrix $g_{n}\in\mathrm{SL}_{d}(\K_{n})$
which is conjugate to $h_{n}$. Then $\theta(g_{n})=\theta(h_{n})\to0$.
Since $g_{n}$ is diagonal, it is clear from the definition of $\theta$
that $g_{n}\to z\mathrm{Id}$ for some $z\in\K_{n}$. Since $\det g_{n}=1$,
then $z^{d}=1$. Moreover, the spectrum of $h_{n}$ must be invariant
under the absolute Galois group of the field extension $\K\leq\K_{h}$,
which implies that $z\in\K$. We will show that $h_{n}=z\mathrm{Id}\in\mathrm{Z}(\Gamma)$
for almost all $n$. By replacing $h_{n}$ with $z^{-1}h_{n}$, we
may assume that $z=1$ and prove that $h_{n}=\mathrm{Id}$ for almost
every $n$. 

We will now work over the field $\tilde{\K}$ defined as the completion
of the algebraic closure of $\K$. In the Archimedean case, $\tilde{\K}=\C$
and in the $p$-adic case the resulting field is usually denoted as
$\C_{p}$. In any case, we may assume that $\K_{n}\subset\tilde{\K}$
for each $n$. Let $X_{\K}=G/K$ be the symmetric space or the Bruhat-Tits
building associated with $G$, as explained in Section \ref{sec:Freeness-criteria-for}.
There exists an analogous space $X_{\tilde{\K}}$ on which $\mathrm{SL}_{d}(\tilde{\K})$
acts by isometries \cite{remy2015bruhat}. Since $g_{n}\to1$ we see
that $\mathrm{disp}_{X_{\tilde{\K}}}(h_{n})=\mathrm{disp}_{X_{\tilde{\K}}}(g_{n})\to0$.
Now, $X_{\K}$ is a $G$-invariant closed convex subset of the $\mathrm{CAT}(0)$
space $X_{\tilde{\K}}$ \cite[Theorem 3.26]{remy2015bruhat} and so,
using the projection map $\pi:X_{\tilde{\K}}\to X_{\K}$ (see \cite[Proposition 2.4]{bridson2013metric}),
it follows that $\mathrm{disp}_{X_{\tilde{\K}}}=\mathrm{disp}_{X_{\K}}$,
and in particular $\mathrm{disp}_{X_{\K}}(h_{n})=\mathrm{disp}_{X_{\tilde{\K}}}(h_{n})\to0$.
Thus there exists $x_{n}\in X$ such that $d_{X_{\K}}(h_{n}x_{n},x_{n})\to0$. 

Fix a compact set $F\subset X_{\K}$ such that $\bigcup_{\gamma\in\Gamma}\gamma.F=X_{\K}$.
Write $x_{n}=\gamma_{n}.f_{n}$ for some $\gamma_{n}\in\Gamma$ and
$f_{n}\in F$. Since $F$ is compact, we may assume (upon extracting
a subsequence) that $f_{n}$ converges to some point $f$. Now 
\begin{align*}
d(h_{n}\gamma_{n}.f,\gamma_{n}.f) & \leq d(h_{n}\gamma_{n}.f,h_{n}\gamma_{n}.f_{n})+d(h_{n}\gamma_{n}.f_{n},\gamma_{n}.f_{n})+d(\gamma_{n}.f_{n},\gamma_{n}.f)\\
 & =d(f,f_{n})+d(h_{n}x_{n},x_{n})+d(f_{n},f)\to0.
\end{align*}
It follows that the sequence of elements $\gamma_{n}^{-1}h_{n}\gamma_{n}$
of $\Gamma$ satisfy $d(\gamma_{n}^{-1}h_{n}\gamma_{n}.f,f)\to0$.
But $\Gamma$ acts on $X$ properly discontinuously, so it follows
that $\gamma_{n}^{-1}h_{n}\gamma_{n}.f=f$ for almost all $n$. Hence
$h_{n}$ fixes the point $\gamma_{n}f$, and is thereby elliptic.
Moreover, the cyclic group $\left\langle h_{n}\right\rangle $ is
contained in the stabilizer of the point $\gamma_{n}.f$, which is
a compact group. Since $\Gamma$ is discrete it must be that $h_{n}$
is torsion. Hence $g_{n}$ is torsion. In particular, all of the eigenvalues
of $g_{n}$ are $d$'th roots of unity. Along with the fact that $g_{n}\to1$
this implies that $g_{n}=1$ for almost all $n$, and therefore $h_{n}=1$. 
\end{proof}
\begin{lem}
\label{lem:random=000020orthogonal=000020points}Let $\Gamma\leq\mathrm{SL}(V)$
be a cocompact lattice, and let $([v],[W])\in\mathcal{F}(V)$ be a
random flag (see \ref{eq:flag}). Then for any $\epsilon>0$, and
$h\in\Gamma$ non-central, we have 
\[
\bP\left[d_{P}([hv],[W])<\epsilon\right]<c\|h\|\epsilon^{1/4}
\]
 for some constant $c\geq1$. 
\end{lem}

\begin{proof}
By Lemma \ref{lem:almost-fixed-points} we have that 
\[
\bP\left[d_{P}([hv],[W])<\epsilon\right]<C_{d}(1+\omega(h))\|h\|^{1/4}\epsilon^{1/4}
\]
By Lemma \ref{lem:spectrum=000020of=000020elements=000020in=000020cocompact=000020lattice},
$\inf_{h\in\Gamma\backslash Z(\Gamma)}\theta(h)\geq\theta$ where
$\theta$ is a positive constant depending only on $\Gamma$. The
function $x\mapsto\sqrt{\frac{1+x}{x^{2}}}$ is decreasing for $x>0$,
and therefore for any $h\in\Gamma\backslash\mathrm{Z}(\Gamma)$, 
\[
\omega(h)=\sqrt{\frac{\|h\|+\theta(h)}{\theta(h)^{2}}}\leq\sqrt{\frac{\|h\|+\theta}{\theta^{2}}}
\]
Moreover, since $\det h=1$ then $\|h\|\geq1$. It follows that 
\[
C_{d}(1+\omega(h))\|h\|^{1/4}\leq C_{d}\left(1+\sqrt{\frac{1+\theta}{\theta^{2}}}\right)\|h\|=c\|h\|
\]
for a suitable constant $c$. Therefore 
\[
\bP\left[d_{P}([hv],[W])<\epsilon\right]<c\|h\|\epsilon^{1/4}
\]
 
\end{proof}

\subsection{Existence of a geometric configuration}
\begin{lem}
\label{lem:existence=000020of=000020points=000020and=000020subspace}Let
$\Gamma\leq G=\mathrm{SL}(V)$ a cocompact lattice. For any $r>0$
and for any $\delta_{0}>0$ there exists hyperplanes $W_{+}$, $W_{-}$,
and points $p_{+}\in W_{-}$, $p_{-}\in W_{+}$, such that
\begin{equation}
\min_{h\in B_{\Gamma}^{\circ}(r)}d_{P}\left(h.\left\{ p_{+},p_{-}\right\} ,W_{+}\cup W_{-}\right)>ce^{-\kappa r}
\end{equation}
and at the same time
\begin{equation}
d_{P}(p_{+},p_{-})\geq1-\delta_{0},\quad d_{P}(p_{+},W_{0})\geq1-\delta_{0},\quad d_{P}(p_{-},W_{0})\geq1-\delta_{0}
\end{equation}
 where $W_{0}=W_{+}\cap W_{-}$. Here $c,\kappa>0$ are constants
depending on $\Gamma$ and on $\delta_{0}$. 
\end{lem}

\begin{proof}
Let $(p_{+},W_{-})$ and $(p_{-},W_{+})$ be two randomly and independently
chosen flags in $\mathcal{F}$ as defined in (\ref{eq:flag}). Almost
surely, $W_{0}=W_{+}\cap W_{-}$ is of co-dimension $2$ and $p_{+},p_{-}\notin W_{0}$.
Moreover, the law of $W_{0}$ is the same as that of a uniformly random
element of $\mathrm{Gr}_{d-2}(V)$. We shall now bound the probability
of  a few unwanted scenarios, and then perform a union bound. 

Fix $\epsilon>0$ and $h\in B_{\Gamma}^{\circ}(r)$. Note that $p_{+}\in P$
is a random point which is independent from the choice of $W_{+}$.
Hence, by Lemmas \ref{lem:Lip} and \ref{lem:tubular-neihbothood}
we have
\begin{equation}
\bP\left[d_{P}(hp_{+},W_{+})<\epsilon\right]\leq\bP\left[d_{P}(p_{+},h^{-1}W_{+})<\epsilon\mathrm{Lip}(h)\right]\leq C_{d}e^{4r}\epsilon\label{eq:unwanted3}
\end{equation}
and similarly
\begin{equation}
\bP\left[d_{P}(hp_{-},W_{-})<\epsilon\right]\leq C_{d}e^{4r}\epsilon\label{eq:unwanted4}
\end{equation}
 Then, by Lemma \ref{lem:random=000020orthogonal=000020points} we
have 
\begin{equation}
\bP\left[d_{P}(hp_{-},W_{+})<\epsilon\right]<ce^{r}\epsilon^{1/4}\label{eq:unwanted1}
\end{equation}
 and similarly 
\begin{equation}
\bP\left[d_{P}(hp_{+},W_{-})<\epsilon\right]<ce^{r}\epsilon^{1/4}\label{eq:unwanted2}
\end{equation}
for some constant $c\geq1$ depending on $\Gamma$. In addition, the
probability that 
\begin{equation}
\bP\left[d_{P}(p_{+},p_{-})<1-\delta_{0},\text{ \textit{or} }d_{P}(p_{-},W_{0})<1-\delta_{0},\text{ \textit{or} }d_{P}(p_{+},W_{0})<1-\delta_{0}\right]=q\label{eq:unwanted5}
\end{equation}
 for some constant probability $0<q<1$ depending on $d$ and on $\delta_{0}$. 

Note that $\left|B_{\Gamma}(r)\right|\leq e^{\alpha r}$ for some
constant $\alpha$ depending on $\Gamma$. It follows that the probability
that at least one of the unwanted scenarios (\ref{eq:unwanted3},\ref{eq:unwanted4},\ref{eq:unwanted1},\ref{eq:unwanted2},\ref{eq:unwanted5})
occurs, for at least one of the $e^{\alpha r}$-many elements $h\in B_{\Gamma}^{\circ}(r)$
is at most
\begin{align}
2C_{d}e^{\alpha r}e^{4r}\epsilon+2ce^{\alpha r}e^{r} & \epsilon^{1/4}+q\leq2c'e^{(4+\alpha)r}\epsilon^{1/4}+q\label{eq:unionbound}
\end{align}
The right hand side of (\ref{eq:unionbound}) can be made strictly
less than $1$, if we let $\epsilon=c'e^{-\kappa'r}$ for sufficiently
large $c'$ and $\kappa'$ which depend on $q,\alpha,c$ but not on
$r$. It follows that there exists points $p_{+},p_{-}$ and hyperplanes
$W_{+},W_{-}$ such that 
\[
\min_{h\in B_{\Gamma}^{\circ}(r)}d_{P}\left(h.\left\{ p_{+},p_{-}\right\} ,W_{+}\cup W_{-}\right)\geq c'e^{-\kappa'r}
\]
 and in addition $d_{P}(p_{+},p_{-})$, $d_{P}(p_{+},W_{0})$, $d_{P}(p_{-},W_{0})\geq1-\delta_{0}$. 
\end{proof}
Having the desired geometric configuration, it is left to find a suitable
element $g\in G$ which realizes it. 
\begin{lem}
\label{lem:gramschmidt}For any $\epsilon>0$ there exists $\delta>0$
such that for any basis $v_{1},...,v_{d}$ of $V$ satisfying $d_{P}([v_{i}],[v_{j}])\geq1-\delta$
for all $i\ne j$, there exists $g\in\mathrm{SL}(V)$ with $|g|\leq\epsilon$
such that $gv_{i}=e_{i}$ for $i=1,...,d$. 
\end{lem}

\begin{proof}
Recall the Iwasawa decompostion $\mathrm{SL}_{d}(\K)=KAN$, stating
that any $g\in\mathrm{SL}_{d}(\K)$ can be written uniquely as a product
$g=kan$ where $k\in K$, $a\in A$ (diagonal matrices), and $n\in N$
(upper triangular unipotent matrices). The proof of the statement
follows from a careful look on the process used to obtain this decomposition. 

Indeed, consider first the Archimedean case. Let $v_{1}',...,v_{d}'$
be the orthogonal basis obtained from $v_{1},...,v_{d}$ via the Gram-Schmidt
process, defined inductively by
\[
v_{k}'=v_{k}-\sum_{i=1}^{k-1}\frac{\left\langle v_{k},v_{i}'\right\rangle }{\left\langle v_{i}',v_{i}'\right\rangle }v_{i}'
\]
The transition matrix $n$ which takes $v_{1},...,v_{d}$ to $v_{1}',...,v_{d}'$
is therefore unipotent and in particular $\det n=1$. Let $k\in K$
such that $kv_{i}'=e_{i}$ for $i=1,...,d$, and set $g=kn$. Then
$gv_{i}=e_{i}$, and $|g|=|n|$. It is left to explain why $|n|\to0$
as $d_{P}([v_{i}],[v_{j}])\to1$. The assignment $(v_{1},...,v_{d})\to n$
described above gives rise to map $\Phi$ from the variety of normalized
bases in $V$ (which is identified with $\mathrm{GL}(V)$ and inherits
this topology), to the space of unipotent matrices. This map is continuous,
and that it maps orthogonal bases to the identity matrix. Hence, as
$\max_{i\ne j}d_{P}([v_{i}],[v_{j}])\to1$, we have $\Phi(v_{1},...,v_{d})\to\mathrm{Id}$,
and in particular $\left|\Phi(v_{1},...,v_{d})\right|\to0$. This
shows the desired statement in the Archimedean case. The non-Archimedean
case is obtained in a similar manner, except that the Gram-Schmidt
process is replaced with the analogous procedure, see \cite[Proposition 4.5.2]{bump1998automorphic}. 
\end{proof}
Since $\Gamma\leq G$ is discrete, there exists a small enough radius
$R>0$ so that the set 
\[
\left\{ s\in G:\|s-1\|<R\right\} 
\]
injects through the map $G\to G/\Gamma$. Let $I(\Gamma)$ denote
the supremum over $R>0$ such this ball injects.\footnote{In fact, up to conjugating $\Gamma$ this value depends only on $G$
by Kazhdan-Margulis theorem. } We conclude this section by producing a lower bound for the values
of the geometric function $\psi_{r}$ near $e\in G$. We recall that
the definition of $\psi_{r}$ depends on a fixed choice of regular
element $a_{0}\in\Gamma$. The results thus far were stated in setting
of an arbitrary identification $V\cong\K^{d}$ with the resulting
norm. At this point, we fix an identification so that $e_{1},...,e_{d}$
is an eigenbasis for $a_{0}$.
\begin{prop}
\label{prop:existence=000020of=000020good=000020element=000020g}There
exists positive constants $c$ and $\kappa$ such that for any $r>0$
\[
\max_{|g|\leq\frac{I(\Gamma)}{2}}\psi_{r}(g)\geq ce^{-\kappa r}
\]
\end{prop}

\begin{proof}
Apply Lemma \ref{lem:gramschmidt}, with $\epsilon=\frac{I(\Gamma)}{2}$,
to obtain a suitable $\delta>0$. Let $p_{+},p_{-},W_{0},W_{+},W_{-}$
as provided by Lemma \ref{lem:existence=000020of=000020points=000020and=000020subspace},
with $\delta_{0}=\delta$. Take a basis $v_{1},...,v_{d}\in V$ such
that $[v_{1}]=p_{+}$, $[v_{d}]=p_{-}$ and $\mathrm{Sp}\{v_{2},...,v_{d-1}\}=W_{0}$.
We get an element $g\in G$ with $|g|\leq\frac{I(\Gamma)}{2}$, such
that $[ge_{i}]=[v_{i}]$. We therefore have 
\[
g.\mathrm{Att}(a_{0})=p_{+},\quad g.\mathrm{Att}(a_{0}^{-1})=p_{-},\quad\text{and }g.\left(\mathrm{Rep}(a_{0})\cap\mathrm{Rep}(a_{0}^{-1})\right)=W_{0}
\]
Hence, for the constants $c,\kappa$ provided by Lemma \ref{lem:existence=000020of=000020points=000020and=000020subspace},
we have that
\begin{align*}
\psi_{r}(g) & =\min_{h\in B_{\Gamma}^{\circ}(r)}d_{P}\left(hg\mathrm{Att}^{\pm}(a_{0}),g\mathrm{Rep}^{\pm}(a_{0})\right)=\\
 & =\min_{h\in B_{\Gamma}^{\circ}(r)}d_{P}\left(h.\left\{ p_{+},p_{-}\right\} ,W_{+}\cup W_{-}\right)>ce^{-\kappa r}.
\end{align*}
\end{proof}

\section{Finding an $r$-free element in $\Gamma$\protect\label{sec:find=000020r-free=000020element=000020in=000020Gamma}}

We will use the following exponential mixing theorem. 
\begin{thm}[Exponential mixing]
\label{thm:Kleinboc-Margulis} Let $G$ be a semisimple linear algebraic
group over a local field $\K$ of characteristic $0$. Let $\Gamma\leq G$
an irreducible lattice, and let $a_{t}\in G$ a one-parameter non-compact
subgroup of semisimple elements. Then there exists constants $E>1$,
$\beta>0$ and $l\in\N$ such that for any two compactly supported
smooth functions $f_{1},f_{2}:G/\Gamma\to\R$, and for all $t\in\K$,
\[
\left|\int_{G/\Gamma}f_{1}(a_{t}x)f_{2}(x)d\mu-\int_{G/\Gamma}f_{1}d\mu\cdot\int_{G/\Gamma}f_{2}d\mu\right|\leq Ee^{-\beta|t|}\|f_{1}\|_{l}\|f_{2}\|_{l}
\]
where $\|\cdot\|_{l}$ denotes the $W_{l}^{2}(G/\Gamma)$-Sobolev
norm. 
\end{thm}

\begin{proof}
In the case $\K$ is Archimedean, this exact statement appears as
Corollary 2.4.6 in \cite{kleinbock1996bounded}. In that reference
there is an assumption which requires the quasi-regular representation
$G\actson L_{0}^{2}(G/\Gamma)$ to have a spectral gap, but this is
automatically satisfied because of \cite[Lemma 3]{bekka1998uniqueness}.
The non-Archimedean case is covered in \cite[Theorem 10.2]{benoist2012effective}.
See also \cite[Theorem 3.1 and Corollary 3.2]{lin2024polynomial}
the suitable analytic notions are defined. 
\end{proof}
We will work with sets rather than functions. In the non-Archimedean
case, one may simply consider indicator functions on small compact
open subgroups, because such functions are smooth. Moreover, the Sobolev
norm of the ball of radius $\epsilon$ in $\mathrm{SL}_{d}(\K)$ is
of the form $c\epsilon^{-m}$ for some constants $c$. The same is
true in the the Archimedean case, except one must replace characteristic
functions with smooth bump functions. Since we are working both with
the Haar measure on $\mathrm{SL}(V)$ but also We will apply the following
lemma with $M=\mathrm{SL}(V)$ endowed with the Haar measure and $\R^{d^{2}}\cong\mathrm{End}(\R^{d})$.
\begin{lem}
\label{lem:bump=000020function}Let $M$ be a $k$-dimensional smooth
submanifold of Euclidean space $\R^{D}$, for some $D\geq k$. Assume
$M$ is endowed with some Riemannian metric with corresponding volume
form $\mu$. Then there exists $\epsilon_{0}>0$ so that for any $0<\epsilon<\epsilon_{0}$
there exists a smooth function $f_{\epsilon}:M\to\R_{\geq0}$ such
that 
\begin{enumerate}
\item $f_{\epsilon}(x)=0$ for all $x\in M$ with $d_{\R^{D}}(x,x_{0})>\epsilon$
(here $d_{\R^{D}}$ the standard Euclidean metric on $\R^{D}$)
\item $\int_{M}fd\mathrm{\mu}=1$. 
\item $\|f\|_{l}\leq C_{k}\epsilon^{-k/2-l}$ where $C_{k}$ is a dimensional
constant. 
\end{enumerate}
\end{lem}

\begin{proof}
Let $B_{\epsilon}$ denote the ball of radius about $x_{0}$ inside
$\R^{D}$ with the Euclidean metric. For a sufficiently small $\epsilon_{0}>0$
there exists a local chart $\varphi:B_{2\epsilon_{0}}\to\R^{D}$ which
such that $\varphi(B_{2\epsilon_{0}}\cap M)\subset\R^{k}\times\{0\}^{D-k}$.
This map $\varphi$ is a diffeomorphism onto its image, and as a result
it is bi-Lipschitz on compact subsets. For this reason, we may assume
$M=\R^{k}\subset\R^{D}$ and that $x_{0}=0$. 

Having reduced to this setting, the proof is standard. Fix a non-zero
and non-negative and non-zero function $\tilde{f}\in C^{\infty}(\R)$
that is supported on $[-1,1]$. Set $f_{\epsilon}:\R^{D}\to\R_{\geq0}$
by $f_{\epsilon}(x)=\frac{1}{\epsilon^{k}}\tilde{f}(\frac{1}{\epsilon}d_{\R^{D}}(x,0))$.
Clearly, $f_{\epsilon}(x)=0$ for all $x$ with $d_{\R^{D}}(x,0)>\epsilon$.
Moreover, for a suitable dimensional constant $C_{k}$, we have
\[
\int_{\R^{k}}f_{\epsilon}=C_{k}\int_{0}^{\infty}r^{k-1}f_{\epsilon}(r)dr=C_{k}\epsilon^{-k}\int_{0}^{\infty}r^{k-1}\tilde{f}(\frac{r}{\epsilon})dr=C_{k}\int_{0}^{\infty}r^{k-1}\tilde{f}(r)dr
\]
Thus, $\int_{\R^{k}}f_{\epsilon}$ is some constant independent on
$\epsilon$, and up to renormalizing $\tilde{f}$ we may assume this
constant to be $1$. A straight forward computation of the derivatives
then yields the bound $\|f_{\epsilon}\|_{l}\leq c\epsilon^{-l-k/2}$.
\end{proof}
Recall that, as in Lemma \ref{lem:holder=000020continuitiy}, we say
that $\psi:G\to\R$ Lipschitz continuous if there exists constants
$C>0$ such that $\left|\psi(gs)-\psi(g)\right|<C\|s-1\|$ for all
$g,s\in G$ with $\|s-1\|\leq\frac{1}{2}$. The following proposition
explains how to replace an element $g\in G$ with a large value under
a function $\psi$, with an element $\gamma\in\Gamma$. 
\begin{prop}
\label{prop:cor=000020from=000020Kleinboc=000020margulis}Let $G$
be a semisimple linear algebraic group over a local field $\K$ of
characteristic $0$. Let $\Gamma\leq G$ an irreducible lattice, and
let $a_{t}\in G$ a one-parameter non-compact subgroup of semisimple
elements. There exists a positive constants $\delta_{0},c_{1},c_{2},c_{3}>0$
such that, any Lipschitz\footnote{Holder continuity suffices as well}
continuous function $\psi:G\to\R_{\geq0}$ which is invariant under
right multiplication of $\{a_{t}\}$, there exists an element $\gamma\in\Gamma$
satisfying
\[
\psi(\gamma)\geq\min\left\{ \frac{1}{2}\zeta,2\delta_{0}\mathrm{Lip}(\psi)\right\} ,\quad\text{ and}\qquad|\gamma|\leq\max\{-c_{1}\log\left(c_{2}\zeta\right),c_{3}\},
\]
where $\zeta=\max_{|g|\leq\delta_{0}}\psi(g)$.
\end{prop}

\begin{proof}
For $\delta>0$, denote $U(\delta)=\left\{ g\in G:\|g-1\|<\delta\right\} $.
Let $\delta_{0}>0$ sufficiently small so that $U(\delta_{0})\cdot U(\delta_{0})^{-1}$
is contained in a ball $B_{G}(I(\Gamma))$; the latter injects through
the map $G\to G/\Gamma$, by definition. We denote by $\mu$ a fixed
Haar measure on $G$, and by $\bar{\mu}$ the corresponding probability
measure on $G/\Gamma$. 

Fix $0<\delta\leq\delta_{0}$. By Lemma \ref{lem:bump=000020function}
there exists a smooth function $f:G\to[0,\infty)$, with $\int_{G}fd\mu=1$,
$\mathrm{supp}(f)\subset U(\delta)$ and $\|f\|_{l}\leq c\delta^{-m}$,
for some constant $m$ depending on $l$ and $d$ (in the non-Archimedean
case, take $f$ to be a normalized characteristic function of of a
ball). Let $g_{0}\in G$ with $|g_{0}|\leq\delta_{0}$ such that $\psi(g_{0})=\zeta$.
Consider the shifted function $f^{g_{0}}:g\mapsto f(gg_{0})$ which
satisfies the same properties as $f$ mentioned above except that
it is supported in $U(\delta)g_{0}^{-1}$. Any function $h:G\to\R$
which is supported in the injected ball $B_{G}(R)$ can be naturally
viewed as a function $\overline{h}:G/\Gamma\to\R$, and $\bar{h}$
has the same Sobolev norm as $h$. Explicitly, for any $x\in G/\Gamma$,
set $\overline{h}(x)=h(g)$ if there exists (and thus unique) $g\in B_{G}(R)$
with $x=g\Gamma$, and $\overline{h}(x)=0$ otherwise.

We apply the exponential mixing theorem to the functions $\overline{f}$
and $\overline{f^{g_{0}}}.$ We get that for all $t\in\K$
\[
\left|\int_{G/\Gamma}\overline{f}(a_{t}x)\overline{f^{g_{0}}}(x)d\bar{\mu}-1\right|\leq c^{2}\delta^{-2m}Ee^{-\beta|t|}
\]
 Fix $t_{0}$ such that
\[
\frac{1}{\beta}\log(c^{2}\delta^{-2m}E)<|t_{0}|\leq\pi^{-1}\frac{1}{\beta}\log(c^{2}\delta^{-2m}E)
\]
where $\pi$ is the uniformizer of $\K$. This ensures that 
\[
\int_{G/\Gamma}\overline{f}(a_{t_{0}}x)\overline{f^{g_{0}}}(x)d\bar{m}>0.
\]
In particular, the sets $a_{t_{0}}^{-1}U(\delta)\Gamma$ and $U(\delta)g_{0}^{-1}\Gamma$
(both subsets of $G/\Gamma$) intersect non-trivially. It follows
that there exists $s,s'\in U(\delta$), and $\gamma\in\Gamma$ such
that $a_{t_{0}}^{-1}s=s'g_{0}^{-1}\gamma$. 

We argue that the element
\[
\gamma=g_{0}s^{-1}a_{t_{0}}^{-1}s'
\]
satisfies the desired properties. Since $\psi$ is Lipschitz, with
Lipschitz constant $L=\mathrm{Lip}(\psi)$, and is $\{a_{t}\}$-invariant,
we get that
\begin{equation}
\left|\psi(g_{0}s^{-1}a_{t_{0}}^{-1}s')-\psi(g_{0}s^{-1})\right|\leq L\delta
\end{equation}
\begin{equation}
\left|\psi(g_{0}s^{-1})-\psi(g_{0})\right|\leq L\delta
\end{equation}
Then by the triangle inequality 
\[
\psi(\gamma)\geq\psi(g_{0})-2L\delta
\]

The estimates thus far hold for any $\delta\leq\delta_{0}$. We split
to two cases. If $\psi(g_{0})\geq4L\delta_{0}$, then, setting $\delta=\delta_{0}$
we get that
\[
\psi(\gamma)\geq2L\delta_{0}.
\]
Moreover, is such case $|\gamma|$ is bounded by some constant (depending
on $\delta_{0},c,E,\beta,m$) 
\[
|\gamma|\leq|g_{0}|+|s|+|a_{t_{0}}|+|s'|\leq3\delta_{0}+\left|t_{0}\right|.
\]
Otherwise, $\psi(g_{0})<4L$$\delta_{0}$, and setting $\delta=\frac{\psi(g_{0})}{4L}$
gives
\[
\psi(\gamma)\geq\psi(g_{0})-2L\delta\geq\frac{1}{2}\psi(g_{0})=\frac{1}{2}\zeta
\]
 Moreover,
\[
|\gamma|\leq3\delta_{0}+\left|t_{0}\right|=-c_{1}\log(c_{2}\zeta)
\]
for some positive constants $c_{1},c_{2}$. 
\end{proof}
We are now ready to complete the proof of Theorem \ref{Thm:effective=000020MIF}.
\begin{proof}[Proof of Theorem \ref{Thm:effective=000020MIF} ]
Let $G=\mathrm{PSL}_{d}(\K)$ for a local field $\K$, and let $\Gamma$
be a cocompact lattice. We consider the length function on $G$, given
by $|g|=d_{G/K}(gx,x)$ (see Section \ref{sec:Freeness-criteria-for}).
Since $\Gamma$ acts cocompactly on $X$, the metric induced by this
length is quasi-isometric to any metric induced by a finite generating
set. It is therefore enough to prove the statement with respect to
the length function $|\cdot|$.\footnote{The same conclusion holds for non-uniform lattices if $d\geq3$, due
to the main result of \cite{lubotzky2000word}. }

Let $\delta_{0},c_{1},c_{2},c_{3}>0$ be constants as guaranteed by
Proposition \ref{prop:cor=000020from=000020Kleinboc=000020margulis}.
Fix a regular and very proximal element $a_{0}\in\Gamma$. Then, the
centralizer $A=C_{G}(a_{0})$ is a maximal split torus in $\mathrm{SL}(V)$.
Fix a one-parameter subgroup $a_{t}\in A$. Given $r\in\N$, let $\psi_{r}$
denote the geometric function as given in Definition \ref{def:geometric=000020function}.
$\psi_{r}$ is invariant under $a_{t}$ by Lemma \ref{lem:A-invariance},
and is $4e^{4r}$-Lipschitz continuous by Lemma \ref{lem:holder=000020continuitiy}.
Let $\zeta_{r}=\max_{|g|\leq\delta_{0}}\psi_{r}(g)$. by Proposition
\ref{prop:cor=000020from=000020Kleinboc=000020margulis} there exists
an element $\gamma\in\Gamma$ such that 
\[
\psi(\gamma)\geq\min\left\{ \frac{1}{2}\zeta_{r},8\delta_{0}\right\} ,\quad\text{ and}\qquad|\gamma|\leq\max\left\{ -c_{1}\log\left(c_{2}\zeta_{r}\right),c_{3}\right\} .
\]
Then, by Proposition \ref{prop:existence=000020of=000020good=000020element=000020g},
$\zeta_{r}\geq ce^{-\kappa r}$ for constants $c,\kappa$ independent
of $r$. It follows that 
\[
\psi(\gamma)\geq c_{4}e^{-\kappa r},\text{ and }|\gamma|\leq c_{5}r
\]
for a some suitable constants $c_{4},c_{5}>0$. This finishes the
proof of Theorem \ref{Thm:effective=000020MIF} due to Lemma \ref{lem:restate=000020goal}. 
\end{proof}

\section{selfless reduced $C^{*}$-algebras\protect\label{sec:selfless-reduced--algebras}}

We now conclude the proof of Theorem \ref{Thm:C*-algebras} from the
introduction. 
\begin{proof}[Proof of Theorem \ref{Thm:C*-algebras}]
Let $\Gamma$ be a cocompact lattice in $\mathrm{PSL}_{3}(\K)$ where
$\K$ is a local field of characteristic $0$. By Theorem \ref{Thm:effective=000020MIF}
and Lemma \ref{lem:selfless}, the group $\Gamma$ is selfless. By
the main result in \cite{ramagge1998haagerup,lafforgue}, $\Gamma$
has the rapid decay property. Thus by \cite[Theorem 3.5 and Corollary 3.8]{2024strictcomparisonreducedgroup},
$C_{r}^{*}(\Gamma)$ is selfless. It therefore follows from \cite[Theorem 3.1]{`2023selfless}
that $C_{r}^{*}(\Gamma)$ is simple, that is canonical trace is the
unique $2$-quasitrace, that $C_{r}^{*}(\Gamma)$ has $1$ strict
comparison and stable rank. It then follows from \cite[Proposition 6.3.1]{robert2012classification}
that, up to approximate unitary equivalent, there exists a unique
unital embedding of the Jiang-Su algebra $\mathcal{Z}$ into $C_{r}^{*}(\Gamma)$.
As for the statement on the Cuntz semigroup, this is an accumulation
of several results in the literature and we refer to \cite[\textbackslash{}S1.3]{2024strictcomparisonreducedgroup}
for the exact details. 
\end{proof}
\bibliographystyle{alpha}
\bibliography{ProjectSelfless}

\vspace{0.5cm}

\noindent{\textsc{Department of Mathematics, University of California San Diego, 9500 Gilman Drive, La Jolla, CA
92093, USA}}
\vspace{0.5cm}

\noindent{\textit{Email address:} \texttt{ivigdorovich@ucsd.edu}}

\noindent{\textit{Webpage:} \texttt{https://sites.google.com/view/itamarv}} \\
\end{document}